\documentclass[11pt,twoside]{article}

\usepackage{geometry}                
\usepackage{graphicx}
\usepackage{amsmath,amssymb, amsthm,epsfig,bm}
\usepackage{setspace}
\usepackage{epstopdf}
\usepackage{psfrag}
\usepackage{color,soul}
\usepackage{verbatim}
\usepackage{algorithm,algorithmic}
\usepackage{multirow}
\usepackage{natbib}
\usepackage{xr}
\usepackage{mathrsfs}
\usepackage{cases}
\usepackage[colorlinks=true,citecolor=blue,linkcolor=blue]{hyperref}
\usepackage{enumitem}

\newcommand{\titleshort}{{Group Uncertainty}}
\newcommand{\authorshort}{{Zhou and Min}}
\newcommand{\apptitle}{Supplemental Material}

\setlist{noitemsep, topsep=0pt}
\bibpunct{(}{)}{;}{a}{}{,} 
\usepackage[english]{babel}
\usepackage[utf8]{inputenc}
\usepackage{fancyhdr}
 
\numberwithin{equation}{section}
\theoremstyle{theorem}
\newtheorem{theorem}{Theorem}[section]
\newtheorem{lemma}[theorem]{Lemma}
\newtheorem{proposition}[theorem]{Proposition}
\newtheorem{corollary}[theorem]{Corollary}

\theoremstyle{definition}
\newtheorem{remark}{Remark}
\newtheorem{assumption}{Assumption}

\newcommand{\myappendix}{\section*{\hfil \apptitle \hfil}\appendix}

\newcommand{\bfX}{X}

\newcommand{\bfW}{{W}}

\newcommand{\bfC}{\Psi}
\newcommand{\bfD}{D}

\newcommand{\bfH}{H}

\newcommand{\bfM}{M}

\newcommand{\bfI}{I}

\newcommand{\calM}{\mathcal{M}}

\newcommand{\calG}{\mathcal{G}}

\newcommand{\calL}{\mathcal{L}}

\newcommand{\calP}{\mathcal{P}}

\newcommand{\scrB}{\mathscr{B}}
\newcommand{\scrC}{\mathscr{C}}
\newcommand{\scrD}{\mathscr{D}}

\newcommand{\scrS}{\mathscr{S}}
\newcommand{\scrX}{\mathscr{X}}

\newcommand{\R}{\mathbb{R}}
\newcommand{\N}{\mathbb{N}}
\newcommand{\mbS}{\mathbb{S}}
\newcommand{\Prob}{\mathbb{P}}

\newcommand{\alp}{\alpha}
\newcommand{\sbeta}{\beta^*}

\newcommand{\hbeta}{\hat{\beta}}

\newcommand{\tdbeta}{\tilde{\beta}}
\newcommand{\Dtdbeta}{\Delta\tdbeta}

\newcommand{\hTheta}{\hat{\Theta}}

\newcommand{\hsigma}{\hat{\sigma}}
\newcommand{\bmSigma}{{\Sigma}}

\newcommand{\tddelta}{\tilde{\delta}}
\newcommand{\hdelta}{\hat{\delta}}

\newcommand{\sdelta}{\delta^*}
\newcommand{\shdelta}{\hdelta^*}
\newcommand{\eps}{\epsilon}
\newcommand{\lmd}{\lambda}
\newcommand{\hlmd}{\hat{\lmd}}

\newcommand{\Omg}{\Omega}

\newcommand{\E}{\mathbb{E}}

\newcommand{\calE}{\mathcal{E}}

\newcommand{\defi}{\mathop{:=}}
  
\newcommand{\dnorm}{\mathcal{N}}

\newcommand{\trans}{\mathsf{T}}
\newcommand{\supth}{\text{th}}

\newcommand{\sqn}{\surd{n}}
\newcommand{\sqpj}{\surd{p_j}}
\newcommand{\sqpmax}{\surd{p_{\max}}}

\newcommand{\frn}{\frac{1}{n}}

\DeclareMathOperator*{\argmin}{argmin}

\DeclareMathOperator{\diag}{diag}
\DeclareMathOperator{\sgn}{sgn}

\DeclareMathOperator{\tr}{tr}

\RequirePackage{tikz}

\newcommand{\siglevel}{\gamma}

\newcommand{\myrev}[1]{#1}

\begin{document}

\title{Uncertainty Quantification Under Group Sparsity}
\author{Qing Zhou\thanks{UCLA Department of Statistics}
\thanks{Supported in part by NSF grants DMS-1055286, 
DMS-1308376 and IIS-1546098 (email: zhou@stat.ucla.edu).} \and Seunghyun Min$^*$ 
}
\maketitle

\begin{abstract}

Quantifying the uncertainty in penalized regression under group sparsity is an important open question. 
We establish, under a high-dimensional scaling, 
the asymptotic validity of a modified parametric bootstrap method for the group lasso,
assuming a Gaussian error model and mild conditions on the design matrix and the true coefficients. 
Simulation of bootstrap samples provides simultaneous inferences
on large groups of coefficients.
Through extensive numerical comparisons, we demonstrate that our bootstrap method
performs much better than popular competitors, highlighting its practical utility.
The theoretical result is generalized to other block norm penalization and sub-Gaussian errors,
which further broadens the potential applications.


\end{abstract}

\section{Introduction}\label{sec:intro}

\subsection{Overview and background}

The surge of recent work on statistical inference for high-dimensional models
can be roughly grouped into a few categories. 
The first group of methods quantify the uncertainty in the lasso \citep{Tibshirani96} and its modifications
via subsampling, data splitting, or the bootstrap, such as \cite{Wasserman09}, \cite{Meinshausen09},
\cite{Chatterjee13}, \cite{LiuYu13} and \cite{Zhou14}.
The second category makes inference along the lasso solution path or conducts post-selection inference
via conditional tests, including \cite{Lockhart13}, \cite{Taylor14} and \cite{Lee14}.
Methods in the third category rely on a de-biased lasso to construct confidence intervals
or perform significance tests \citep{ZhangZhang11,vandeGeer13,Javanmard13b}.
In addition, \cite{NingLiu14} and \cite{Voorman14} have proposed different score tests for penalized M-estimators
or penalized regression.
Under certain sparsity assumptions on the parameter space, all the methods make use of regularization, 
particularly $\ell_1$ penalization. 

This paper focuses on statistical inference under group sparsity,
which arises naturally in many applications. Furthermore, individual coefficients that are too small to
detect may be reliably identified when grouped together. This gives another motivation
for the present work. Consider the linear model
\begin{equation}\label{eq:model}
y=\bfX \beta_0 + \varepsilon,
\end{equation}
where $\beta_0\in\R^p$ is the unknown true parameter, $y\in\R^n$ is the response vector, 
$\bfX=(X_1\mid \cdots \mid X_p)\in \R^{n \times p}$ is the design matrix and
$\varepsilon\in\R^n$ is an independent and identically distributed 
error vector with mean zero and variance $\sigma^2$.
Suppose that the predictors are partitioned into $J$ nonoverlapping groups, denoted by
$\calG_j$ for $j=1,\ldots,J$. That is, 
$\cup_{j=1}^J \calG_j=\{1,\ldots,p\}$ and $\calG_j \cap \calG_k= \varnothing$ for every $j \ne k$.
For $\beta=(\beta_1,\ldots,\beta_p)$, let $\beta_{(j)}=(\beta_k)_{k\in\calG_j}$ for each $j$.
The group lasso \citep{YuanLin06} is then defined as
\begin{equation}\label{eq:grplassodef}
\hbeta \in \argmin_{\beta\in\R^p} \left\{\frac{1}{2} \|y-\bfX \beta \|^2 + 
n \lambda \sum_{j=1}^J w_j \| \beta_{(j)}\|\right\},
\end{equation}
where $\|\cdot\|$ denotes the Euclidean norm.
The weights $w_j >0$ and are usually in proportion to $\sqpj$, 
where $p_j = |\calG_j|$ is the size of the $j$th group. In \eqref{eq:grplassodef},
we use $\in$ to indicate that $\hbeta$, when not unique, is one of the minimizers.
Given the group structure $\calG=(\calG_1,\ldots,\calG_J)$ and $\alpha \in[1,\infty]$, 
the associated $(1,\alpha)$-group norm, sometimes referred to as the $\ell_{1,\alp}$ norm, is defined as
\begin{equation}\label{eq:groupnorm}
\|\beta\|_{\calG,\alpha}=\sum_{j=1}^J \| \beta_{(j)}\|_{\alpha},
\end{equation}
where $\|\cdot\|_{\alp}$ denotes the $\ell_{\alp}$ norm of a vector.
Thus, the penalty term in \eqref{eq:grplassodef} is a weighted $(1,2)$-group norm of $\beta$.
One may also use other $(1,\alpha)$-group norm penalties in the above formulation \citep{Negahban12}.
See \cite{Huang12} for recent developments on regularization methods respecting group structure.
Despite the wide applications of these methods, there are quite limited efforts
devoted to uncertainty quantification and inference under group sparsity.
\cite{MitraZhang14} propose to de-bias the group lasso for inference about a group of coefficients,
which generalizes the idea of those methods in the third category reviewed above.
For each group, a highly nontrivial optimization problem has to be solved, and thus it might be difficult
to apply this approach to a large number of groups.
\cite{Meinshausen15} develops a conservative group-bound method to detect whether a group of
highly correlated variables have nonzero coefficients. 

\subsection{Contributions of this work}\label{sec:contributions}

In this article, we tackle the problem of group inference from a different angle. 
Our goal is to directly quantify the uncertainty in the group lasso \eqref{eq:grplassodef} 
and other penalized estimators under group sparsity. Neither of the two 
aforementioned group inference methods provides an answer to this question. 
Instead, we consider the parametric bootstrap, a simple simulation-based approach.
Given a proper choice of a point estimator $\tdbeta$, we simulate an error vector $\varepsilon^*$ and
put $y^*=\bfX\tdbeta + \varepsilon^*$. After minimizing the penalized loss \eqref{eq:grplassodef}
we obtain a group lasso solution $\sbeta$ for the simulated data $y^*$. 
Under quite mild conditions on the
design matrix $\bfX$ and the true parameter $\beta_0$, we provide theoretical justification for 
using the distribution of $(\sbeta-\tdbeta)$, conditional on $\tdbeta$, 
to make inference about $\beta_{0(j)}$ for all $j$.
Our theory is developed under a high-dimensional asymptotic framework as $J\gg n\to\infty$
and applies to large groups with size $p_j\to\infty$. 
To the best of our knowledge, such consistency in estimation of sampling distributions has not been
established for group norm penalized estimators. 
Allowing for unbounded group sizes makes the group lasso
fundamentally different from the lasso, 
so this work represents a distinct contribution from existing bootstrap methods
for a lasso-type estimator \citep{Chatterjee13,Zhou14}. It is also different from the work by
\cite{McKeague15} and by \cite{Shah15}, in which the bootstrap is used for correlation screening
or simulation of a scaled residual without considering any group structure.

In addition to the novel theoretical result, 
an important strength of this work is its great potential in practical applications, as seen from the following aspects.
\begin{enumerate}
\item {\em Flexibility}: By simulation one can easily estimate the distributions of many functions
of $\hbeta$, and thus has much more freedom 
in choosing which statistic to use in an inference problem. The simulation approach allows for
interval estimation and significance tests for all groups simultaneously and in fact for all individual coefficients as well. 
This is in sharp contrast to the de-biased methods which solve an optimization
problem for each individual coefficient or each coefficient group.

\item {\em Implementation}: There is no need for any additional optimization algorithm. 
For most cases, simple thresholding of the group lasso is a valid choice for the point estimate $\tdbeta$.
Therefore, a practitioner can simply
use the same software package to find the group lasso solution and to quantify its uncertainty.
According to our empirical study, a few hundred bootstrap samples are usually sufficient 
for accurate inference,
which as a whole cost much less time than solving $p$ optimization problems 
as used in the de-biased lasso approach. 

\item {\em Performance}: Through extensive numerical comparisons, we demonstrate that
our bootstrap method outperforms competing methods by a large margin for finite samples, which implies that this method
is much less dependent on asymptotic approximation. Moreover, our method is not very sensitive
to the threshold value used to define $\tdbeta$, which is the only user-input parameter in our current implementation.
\end{enumerate}
Although out of the scope of this paper, the method of estimator augmentation \citep{Zhou14}
may greatly improve the simulation efficiency, particularly, in 
calculating tail probabilities in a significance test,
making a simulation-based method more appealing in applications.
This has been demonstrated for group inference by \cite{Zhou16}.

\subsection{Organization and notation}

The paper is organized as follows. Our parametric bootstrap method for inference with the group lasso is
proposed and described in Section~\ref{sec:method}. 
We develop asymptotic theory under a Gaussian error distribution in Section~\ref{sec:theory}
to show that our inferential method is valid in a high-dimensional framework.
Numerical results are provided in Section~\ref{sec:numerical} to demonstrate the advantages
of our method in finite-sample inference over competing methods.
In Section~\ref{sec:generalizations}, we generalize our results to the use of the
$(1,\alpha)$-group norm penalty \eqref{eq:groupnorm}, for $\alpha\in(2,\infty]$,
and to sub-Gaussian errors, with a discussion on future work.
All proofs are deferred to 
Supplemental Material which also contains auxiliary theoretical and numerical results.

Notation used throughout the paper is defined here. Define $\N_k=\{1,\ldots,k\}$ for an integer $k\geq 1$. 
Let $A\subset \N_m$ be an index set. 
For a vector $v=(v_j)_{1:m}$, denote by $v_A=(v_j)_{j\in A}$ the restriction of $v$ to the components in $A$.
For a matrix $\bfM=(M_{ij})_{n\times m}$ with
columns $M_j$, $j=1,\ldots,m$, define $\bfM_A=(M_{j})_{j\in A}$ as a matrix of size $n\times |A|$
consisting of columns in $A$, and similarly define $\bfM_{BA}=(M_{ij})_{i\in B, j\in A}$ for $B\subset \N_n$.
Denote by $\diag(v)$ the $m\times m$ diagonal matrix with $v$ as the diagonal elements
and by $\diag(\bfM,\bfM')$ the block diagonal matrix with $\bfM$ and $\bfM'$ as the diagonal blocks.
For a square matrix $\bfM\in\R^{m\times m}$, 
$\diag(\bfM)$ extracts the diagonal elements, $\tr(\bfM)$ denotes the trace of $\bfM$, 
and $\Lambda_k(\bfM)$, for $k\in\N_m$, denotes its eigenvalues. 
Moreover, $\Lambda_{\max}(\bfM)$ and $\Lambda_{\min}(\bfM)$ denote the maximum and the minimum eigenvalues,
respectively, or the supremum and the infimum when $m\to\infty$. 
Denote by $\bfI_m$ the $m\times m$ identity matrix.
Given the group structure $\calG$, let $\calG_A = \cup_{j\in A} \calG_j\subset \N_p$ for $A \subset \N_J$.
For a vector $v=(v_j)_{1:p}$, define $v_{(A)} = v_{\calG_A}$
and, in particular, $v_{(j)}=v_{\calG_j}$. 
We call $G(v)=\{j\in\N_J: v_{(j)}\ne 0\}$ the set of active groups of $v$.
For an $m \times p$ matrix $\bfM$, $\bfM_{(A)}=\bfM_{\calG_A}$, and when $m=p$,
$\bfM_{(AB)}=\bfM_{\calG_A \calG_B}$ for $B \subset \N_J$. 
For two sequences $a_n$ and $b_n$, write $a_n=\Omega(b_n)$ if $b_n=O(a_n)$
and $a_n\asymp b_n$ if $a_n=O(b_n)$ and $b_n=O(a_n)$.
Their probabilistic counterparts are written as $\Omega_P$ and $\asymp_P$, respectively. 
We use $\nu[Z]$ to denote the distribution of a random vector $Z$.
Positive constants $c_1$, $c_2$, etc. and positive integers $N_1$, $N_2$, etc. are defined locally
and may have different values from line to line.
Small positive numbers are often denoted by $\epsilon$, which should be distinguished from the error
vector $\varepsilon$.

\section{Bootstrapping the group lasso}\label{sec:method}

\subsection{Bootstrap and inference}\label{sec:pbinference}

We will assume that the noise vector $\varepsilon\sim\dnorm_n(0,\sigma^2\bfI_n)$
throughout the paper until Section~\ref{sec:generalizations}, in which
sub-Gaussian and other error distributions are considered.
Let $\tdbeta$ and $\hsigma$ be point estimates of $\beta_0$ and $\sigma$.
We first describe our proposed parametric bootstrap for the group lasso, 
given point estimates $\tdbeta$ and $\hsigma$, and discuss how to make
inference with a bootstrap sample. In next subsection, we will propose methods to construct
$\tdbeta$ and $\hsigma$ from the data $(y,\bfX)$.

Let $B(y;\lambda)$ denote the set of minimizers of the loss in \eqref{eq:grplassodef}
so that $\hbeta\in B(y;\lambda)$.
Our parametric bootstrap for the group lasso contains two steps:
\begin{enumerate}
\item[(1)] Given $(\tdbeta,\hsigma)$, draw $\varepsilon^*  \sim \dnorm_n(0,\hsigma^2\bfI_n)$ 
and set $y^*=\bfX \tdbeta + \varepsilon^*$;
\item[(2)] Solve \eqref{eq:grplassodef} with $y^*$ in place of $y$ to obtain $\sbeta\in B(y^*;\lambda)$.
\end{enumerate}
After drawing a large sample of $\sbeta$ values via the above procedure, we can make inference
for each group $j\in\N_J$. By default, we choose the function
\begin{equation}\label{eq:inferfunction}
f_j(\sbeta_{(j)}-\tdbeta_{(j)})=\|\bfX_{(j)}(\sbeta_{(j)}-\tdbeta_{(j)})\|^2
\end{equation}
to build a confidence region and carry out a significance test for $\beta_{0(j)}$.
From the bootstrap sample of $\sbeta$,
we estimate the $(1-\siglevel)$-quantile $f_{j,(1-\siglevel)}$ such that
\begin{align*}
\Prob\left\{ \left. f_j(\sbeta_{(j)}-\tdbeta_{(j)})>f_{j,(1-\siglevel)} \right| \bfX,\tdbeta,\hsigma\right\} = \siglevel.
\end{align*}
Then our $(1-\siglevel)$ confidence region for $\beta_{0(j)}$ is 
\begin{align}\label{eq:regionest}
R_{j}(\siglevel)=\left\{\theta\in\R^{p_j}:f_j(\hbeta_{(j)}-\theta)\leq f_{j,(1-\siglevel)}\right\}.
\end{align}
We may also test the hypothesis $H_{0,j}:\beta_{0(j)}=0$,
which will be rejected at level $\siglevel$ if $0 \notin R_{j}(\siglevel)$, that is, if
$\|\bfX_{(j)}\hbeta_{(j)}\|^2> f_{j,(1-\siglevel)}$.
This approach can make inference about $\beta_{0(j)}$ simultaneously for $j=1,\ldots,J$.
One may choose other matrices in place of $\bfX_{(j)}$ to define $f_j$ \eqref{eq:inferfunction}, 
as long as they satisfy some very mild conditions specified in Corollary~\ref{thm:knownvarlaw}. 

\subsection{Estimation of parameters}\label{sec:parameter}

Multiple methods may be applied to obtain the point estimates, $\tdbeta$ and $\hsigma$.
Roughly speaking, $\tdbeta$ must consistently recover
the groups of $\beta_0$ with a large $\ell_2$ norm and $\hsigma$ needs
to be a consistent estimator of $\sigma$ with a certain convergence rate.

One possible way to construct $\tdbeta$ is to threshold the group lasso $\hbeta$,
\begin{equation}\label{eq:thresholdhbeta}
\tdbeta_{(j)}=\hbeta_{(j)}I(\|\hbeta_{(j)}\|>b_{\supth}), \quad j=1,\ldots,J,
\end{equation}
where $b_{\supth}>0$ is a cutoff value. Useful practical guidance
is to choose the cutoff so that all small coefficient groups will be thresholded to zero. 
Let $A=G(\tdbeta)$ be the active groups of $\tdbeta$ and $\hat{M}=\calG_{A}$
be the set of active coefficients of $\tdbeta$. Then we perform a least-squares regression
of $y$ on $\bfX_{\hat{M}}$ to re-calculate the nonzero coefficients of $\tdbeta$, which reduces their bias, 
and to estimate the error variance $\hsigma^2$, provided that $|\hat{M}|<n$.
This will be implemented in our method for the numerical comparisons with more details
provided in Section~\ref{sec:simulationsettings}.
See Section~\ref{sec:parametertheoretical} for other choices of the point estimates
and theoretical justifications.

The parametric bootstrap is commonly used for fixed-dimensional inference problems, 
such as linear regression with $p$ fixed and $n\to\infty$. However, there is no general theory on its validity for  
a high-dimensional problem. 
Thus, rigorous asymptotic theory for our bootstrap method, including the choice of the point estimators, 
will be developed in Section~\ref{sec:theory} to justify its use 
under a setting that allows $J\gg n\to\infty$ and $p_{j}\to\infty$.
If $\varepsilon^*$ is drawn by resampling the residual $\tilde{\varepsilon}=y-\bfX \tdbeta$, 
then our method implements the standard residual bootstrap. 
To the best of our knowledge, consistency of either bootstrap method for the group lasso
under the above high-dimensional setting has not been established in the literature.

\myrev{When the group size $p_j=1$ for all $j$, our method reduces to a parametric bootstrap for
the lasso. In this special case, it is closely related to the modified residual bootstrap proposed by \cite{Chatterjee11}. 
However, the consistency of their method is established under the assumption that $p$ is fixed,
while our theory applies to the more interesting case of $p\gg n$.}

\section{Asymptotic theory}\label{sec:theory}

\subsection{Convergence of the bootstrap distribution}\label{sec:knownvar}

Let $A_0=G(\beta_0)$, $q_0=|A_0|$ and $s_0=|\calG_{A_0}|$ so that $q_0$ is the number of active groups
and $s_0$ the number of active coefficients of $\beta_0$.
Denote by $p_{\min}$ and $p_{\max}$ the minimum and the maximum of $\{p_j\}$, respectively.
Assume that $w_j\in[w_*,w^*]$ for all $j\in\N_J$ with $0<w_*<w^*<\infty$ throughout this section.
We adopt a high-dimensional asymptotic framework for model \eqref{eq:model}, 
where $p=p(n) > n \to \infty$, $q_0,p_{j}$ may be unbounded, and $\sigma^2$ stays as a constant. Accordingly,
$y$, $\bfX$, and $\beta_0$ all depend on $n$. 
For brevity the index $n$ is often suppressed. 
We say a sequence of events $\calE_n$ happens with high probability if $\Prob(\calE_n)\to 1$
as $n\to\infty$.

We first develop theoretical results assuming that the noise variance $\sigma^2$ is known, in which case
we let $\hsigma=\sigma$ in our bootstrap sampling of $\sbeta$. 
To facilitate our analysis, we introduce an intermediate variable
$\hbeta^*\in B(\bfX\beta_0+\varepsilon^*;\lambda)$, which follows the same distribution as $\hbeta$
since $\nu[\varepsilon^*]=\nu[\varepsilon]$. Define centered and rescaled estimators
\begin{equation}\label{eq:defalldelta}
\hdelta= r_n(\hbeta-\beta_0), \quad\quad \shdelta=r_n(\hbeta^*-\beta_0),\quad\quad \sdelta= r_n(\sbeta-\tdbeta),
\end{equation}
where $r_n$ is a sequence of positive numbers to be specified later.
We will show that the $\ell_2$ deviation between $\shdelta$ and $\sdelta$, conditioning on a proper choice of $\tdbeta$,
converges to zero in probability, 
which leads to weak convergence of functions of  $\sdelta$ to $\hdelta$. 
The case of unknown $\sigma$ will be covered in Section~\ref{sec:parametertheoretical}.

We assume the following conditions on $\bfX$ and $\beta_0$. 
Denote by $\bfC=\frn X^{\trans}X$ the Gram matrix. 
Let $\alpha\in[1,\infty]$, $Z\sim\dnorm_p(0,\bfI_p)$ and 
\begin{equation*}
\rho(\alpha^*)=\E\left\{\max_{j\in\N_J} \|Z_{(j)}\|_{\alpha^*}/\sqn\right\}, 
\end{equation*}
where $\alpha^*$ is conjugate to $\alpha$ satisfying $\frac{1}{\alpha}+\frac{1}{\alpha^*}=1$. 

\begin{assumption}\label{as:design}
With high probability,  we have $\Lambda_{\max}(\bfC_{(jj)})\leq \bar{c}<\infty$ for all $j\in\N_J$ and
\begin{equation}\label{eq:groupstrconvexmain}
\frac{1}{n}\|\bfX\Delta\|^2 \geq \kappa_1 \|\Delta\|^2 - \kappa_2 \rho^2(\alpha^*) \|\Delta\|^2_{\calG,\alpha}
\quad\text{ for all } \Delta\in\R^p,
\end{equation}
where $\kappa_1,\kappa_2>0$ are universal constants.
\end{assumption}

\begin{assumption}\label{as:betasignal}
The true coefficient vector $\beta_0$ is sparse:
\begin{equation}\label{eq:sparsescaling}
{q_0 (p_{\max} \vee \log J)}/{\sqn} =o(1),
\end{equation}
and $A_0=S_1 \cup S_2$ such that
\begin{equation}\label{eq:separablebeta}
\inf_{j\in S_1}\frac{\|\beta_{0(j)}\|}{\sqpj} \gg \frac{b_n}{\sqn}
 \quad\text{and}\quad \sup_{j\in S_2}\frac{\|\beta_{0(j)}\|}{\sqpj} \ll \frac{1}{b_n\sqn},
\end{equation}
where $b_n=s_0 \{1\vee {(\log J/p_{\max}})^{1/2}\}$.
\end{assumption}

Our next assumption is on the point estimator $\tdbeta$. For $c>0$, define 
\begin{equation}\label{eq:defdeltaregion}
\scrD(c)=\{\delta\in\R^p: \|\delta_{(j)}\|\leq c \sqpj\quad \text{for all } j \in A_0\}. 
\end{equation} 
\begin{assumption}\label{as:pluginest}
The estimator $\tdbeta=\tdbeta(y,\bfX)$ satisfies $\Prob\{\tdbeta\in\scrB(M_1)\}\to 1$ for some $0<M_1<\infty$,
where the set $\scrB(M_1)$ is defined by
\begin{equation}\label{eq:BM2}
\scrB(M_1)=\{\beta\in\R^p: G(\beta)\subset S_1 \text{ and } r_n(\beta-\beta_0) \in \scrD(M_1)\}.
\end{equation}
\end{assumption}

\myrev{The inequality \eqref{eq:groupstrconvexmain} in Assumption~\ref{as:design} guarantees the 
restricted eigenvalue assumption \citep{Bickel09,Lounici11} under group sparsity,
which is shown in Supplemental Material.
It is also sufficient for the restricted strong convexity condition for the group lasso as demonstrated in \cite{Negahban12}.}
Assumption~\ref{as:betasignal} shows that $\beta_0$ is sparse at both group and individual levels,
and its active groups can be separated into strong and weak signals. 
A beta-min condition \eqref{eq:separablebeta} is satisfied by the strong signal groups.
As implied by \eqref{eq:BM2}, the point estimator $\tdbeta$ must identify only the strong groups
and converge to the active groups of $\beta_0$ at certain rate. 
A detailed discussion of these assumptions with comparisons
to existing methods is provided in Section~\ref{sec:theorycomparison}.

\begin{theorem}\label{thm:asympknownvar}
Consider the model \eqref{eq:model} with $\varepsilon \sim \dnorm_n(0,\sigma^2\bfI_n)$ 
and $p_{\max}/p_{\min}\asymp 1$.
Suppose that Assumptions \ref{as:design} to \ref{as:pluginest} hold.
Choose a suitable $\lambda \asymp \{(p_{\max} \vee \log J)/n\}^{1/2}$. 
Let $r_n$ be a positive sequence such that
\begin{equation}\label{eq:defrn}
(1/\sqpmax)\sup_{j\in A_0}\|\hbeta_{(j)}-\beta_{0(j)}\|\asymp_P r_n^{-1}
\end{equation}
for some $\hbeta$ defined in \eqref{eq:grplassodef}. 
Put $\hdelta=r_n(\hbeta-\beta_0)$ and $\tddelta=r_n(b-\beta_0)$, where $b$ is any minimizer of \eqref{eq:grplassodef}.
Define $\shdelta$ and $\sdelta$ by \eqref{eq:defalldelta} with $\varepsilon^* \sim \dnorm_n(0,\sigma^2\bfI_n)$. 
Then for every $\epsilon>0$,
\begin{align}
&\Prob(\|\hdelta-\tddelta\| > \epsilon \mid \bfX)=o_p(1), \label{eq:difftwohdelta} \\
&\Prob\left\{\|\sdelta-\shdelta\| > \epsilon \mid \bfX,\tdbeta\right\}=o_p(1). \label{eq:DeltatoinP0} 
\end{align}
\end{theorem}

For $v\in\R^{p_j}$, $\|v\|/\sqpj$ can be regarded as an average length of its components.
Using such averages serves as a normalization across groups of different sizes.
Since $p_{\max}/p_{\min}\asymp 1$, the $r_n^{-1}$ in \eqref{eq:defrn} is the convergence rate of the supremum of
the normalized $\ell_2$ errors for the active groups. 
As shown in the proof of Theorem~\ref{thm:asympknownvar} in Supplemental Material,
\begin{equation}\label{eq:rateboundsofrn}
\Omg(\lambda^{-1}({p_{\max}/q_0})^{1/2}) = r_n =O(\sqn).
\end{equation}
In Remark~\ref{rm:anorderbound} below, we will see that $r_n$ may achieve the
optimal rate of $\sqn$ under certain conditions.
The key conclusion \eqref{eq:DeltatoinP0} shows that the $\ell_2$ deviation
between $\sdelta$ and $\shdelta$ converges to zero. 
We do not assume that \eqref{eq:grplassodef} has a unique minimizer in Theorem~\ref{thm:asympknownvar}. However,
\eqref{eq:difftwohdelta} implies that any two minimizers follow the same asymptotic distribution
and thus there is no need to distinguish these minimizers when $n$ is large.

We briefly comment on some technical aspects in the proof of this result,
while leaving the detail to Supplemental Material.
The overall idea is to bound the difference in the group lasso loss function \eqref{eq:grplassodef}
when $y^*$ is used in place of $y$ and then translate this into
a bound on the deviation between $\sdelta$ and $\shdelta$, which are centered minimizers of the loss.
The challenge in the first step comes from the non-linearity of the regularizer and the high-dimensionality of 
the space. As a result, we cannot restrict our analysis to any finite-dimensional compact subset.
Although $\shdelta$ and $\sdelta$ satisfy the so-called cone condition which allows one to make use
of restricted eigenvalue assumptions on $\bfX$, the deviation $(\sdelta-\shdelta)$ does not necessarily
lie in a cone and is usually not sparse. This presents another technical challenge in the second step.


Now consider the implications of Theorem~\ref{thm:asympknownvar} for 
inference about a coefficient group. Theorem 5.1 in \cite{Lounici11} asserts that
\begin{equation}\label{eq:groupwisebound}
\sup_{j\in\N_J} \|\hbeta_{(j)}-\beta_{0(j)}\| =O_p(\lambda)
\end{equation}
under a generalized coherence condition for the group lasso setting. 
Let
\begin{equation}\label{eq:defan}
a_n=\lambda^{-1}(p_{\max})^{1/2}=\left\{{p_{\max}}/({p_{\max}\vee \log J})\right\}^{1/2}\sqn,
\end{equation}
$T_j={L_j(\hbeta_{(j)}-\beta_{0(j)})}$, and $T^*_j={L_j(\sbeta_{(j)}-\tdbeta_{(j)})}$,
where $L_j$ is a matrix for each $j$. Recall that $\nu[Z]$ is the distribution of $Z$.
We have the following result regarding group inference: 

\begin{corollary}\label{thm:knownvarlaw}
Consider the model \eqref{eq:model} with $\varepsilon \sim \dnorm_n(0,\sigma^2\bfI_n)$ 
and $p_{\max}/p_{\min}\asymp 1$.
Suppose that Assumptions \ref{as:design} to \ref{as:pluginest} are satisfied, and
that \eqref{eq:groupwisebound} holds for a suitable $\lambda \asymp \{(p_{\max} \vee \log J)/n\}^{1/2}$.  
Assume that the singular values of  $L_j$
are bounded from above by a positive constant uniformly for all $j\in\N_J$.
Then there are $\Delta^*_j\in\R^{p_j}$ such that
\begin{align}
&\nu[a_n T^*_j +  \Delta^*_j \mid \bfX, \tdbeta] =\nu[a_nT_j\mid \bfX]
\quad \text{ for all $j\in\N_J$}, \label{eq:Mnormweakconv} \\
&\Prob\left(\sup_{j\in\N_J}\left.\|\Delta_j^*\| > \epsilon \,\right| \bfX,\tdbeta\right)=o_p(1)
\end{align}
for every $\epsilon>0$. Furthermore, $a_n\asymp \sqn$ if $p_{\max}=\Omg(\log J)$.
\end{corollary}

The above result applies directly to simultaneous inference about all groups, each of a possibly unbounded size,
based on the bootstrap distributions of $\|T^*_j\|$.
Under mild assumptions such as those in Proposition~\ref{prop:GaussianC1} below, 
the singular values of $\bfX_{(j)}/\sqn$ are uniformly bounded between two positive constants with high probability.
Thus, this corollary validates theoretically our bootstrap inference using the
function $f_j$ \eqref{eq:inferfunction}. In fact, one may use many other matrices to carry out the inference.
Moreover, the explicit
rate of $a_n$ is irrelevant to the practical implementation of our method.

\begin{remark}\label{rm:anorderbound} 
The order of $a_n$ in \eqref{eq:defan} shows that the radius of a confidence region constructed
according to \eqref{eq:Mnormweakconv} 
is wider than the optimal rate by no more than a factor of $1\vee{(\log J/p_{\max})^{1/2}}$.
It should be clarified that this suboptimality is caused by the intrinsic bias of 
the group lasso, instead of the bootstrap procedure. When the order of the group size 
is comparable to or larger than $\log J$,
however, the rate may become optimal with $a_n\asymp\sqn$.
This is not possible for the lasso with $p_{\max}=1$ and $J=p\gg n\to \infty$, which demonstrates
another advantage of grouping a large number of coefficients
for inference with the group lasso. Comparing \eqref{eq:defrn} to \eqref{eq:groupwisebound} 
shows that $r_n=\Omg(a_n)$ in general. 
As a result, when $a_n$ becomes optimal we also have $r_n\asymp \sqn$.
\end{remark}

\subsection{Justification for point estimates}\label{sec:parametertheoretical}

To complete our validation of the proposed bootstrap inference in Section~\ref{sec:method} 
assuming $\sigma$ is known, it remains to verify that the thresholded group lasso $\tdbeta$ 
satisfies Assumption~\ref{as:pluginest}:

\begin{proposition}\label{prop:bth}
Consider the model \eqref{eq:model} with $\varepsilon \sim \dnorm_n(0,\sigma^2\bfI_n)$ 
and $p_{\max}/p_{\min}\asymp 1$. Suppose that Assumptions \ref{as:design} and \ref{as:betasignal} hold.
Choose a suitable $\lambda \asymp \{(p_{\max} \vee \log J)/n\}^{1/2}$. 
Define $\tdbeta$ by \eqref{eq:thresholdhbeta} and assume that 
\begin{equation}\label{eq:orderofbth}
\lmd s_0  \gg b_{\supth} \gg  \lambda(q_0)^{1/2}.
\end{equation} 
Then $\Prob\{\tdbeta\in\scrB(M_1)\} \to 1$ for some $M_1<\infty$.
\end{proposition}

Under Assumption~\ref{as:betasignal} with $S_2=\varnothing$,
the condition that $\Prob(\tdbeta\in\scrB(M_1)) \to 1$ requires $\tdbeta$ to be model selection consistent
and to have a certain rate of convergence to $\beta_0$. Of the two, model selection consistency is
the key, since conditional on $G(\tdbeta)=A_0$ and $s_0\ll \sqn$ \eqref{eq:sparsescaling}, one can always apply 
the ordinary least-squares method to reestimate the active coefficients of $\beta_0$, which will satisfy
the convergence rate requirement. Here,
we mention a few other methods for constructing $\tdbeta$, which are model selection consistent.
The first method is the adaptive group lasso \citep{WeiHuang10}, a natural generalization
of the adaptive lasso \citep{Zou06}, 
which minimizes \eqref{eq:grplassodef} with weights $w_j$ defined by an initial estimator.
If we choose the group lasso as the initial estimator, Corollary 3.1 in \cite{WeiHuang10}
asserts that the adaptive group lasso is model selection consistent, 
while allowing $J$, $q_0$, and $p_{\max}$ to grow with $n$.
The second choice is to use a concave penalty, such as the MCP \citep{Zhang10}.
By Theorem 4.2 and Corollary 4.2
in \cite{Huang12}, a penalized estimator under the group MCP 
enjoys the oracle property, achieving model selection consistency
and the optimal rate in estimating active coefficients.
Other possible methods may include stability selection \citep{Meinshausen10} and 
the sample splitting approach of \cite{Wasserman09}.

When the noise variance is unknown, we plug in an estimate $\hsigma$ in bootstrap sampling.
To establish similar results as in Theorem~\ref{thm:asympknownvar} and Corollary~\ref{thm:knownvarlaw}, 
it suffices that
\begin{equation}\label{eq:asymcondhsigma}
q_0 (p_{\max} \vee \log J) \left\{(\hsigma/\sigma) \vee (\sigma/\hsigma) - 1 \right\}=o_p(1).
\end{equation}
See Supplemental Material for the precise statement.
This essentially requires $\hsigma$ converge to $\sigma$ faster than the rate
of $\{q_0 (p_{\max} \vee \log J)\}^{-1}$.
If $\hsigma$ is $\sqn$-consistent, then for \eqref{eq:asymcondhsigma} to hold it is sufficient to have
${q_0 (p_{\max} \vee \log J)}\ll{\sqn}$,
which does not impose any additional assumption on the scaling among $(n,J,q_0,p_{\max})$
beyond the one in \eqref{eq:sparsescaling}.
There are a few possible approaches that can achieve this desirable convergence rate. 
One may employ a two-stage approach, which selects a model $\hat{M}$ in the first stage
and then estimates the error variance by ordinary least-squares using only the variables in $\hat{M}$.
For our bootstrap approach, we let $\hat{M}=\calG_{A}$, where $A=G(\tdbeta)$ is the set of active groups 
of the thresholded group lasso $\tdbeta$ defined in \eqref{eq:thresholdhbeta}. 
For a proper choice of the cutoff value $b_{\supth}$ \eqref{eq:orderofbth},
$\hat{M}=\calG_{S_1}$ with high probability, selecting the strong coefficients consistently.
Then the two-step approach leads to a $\sqn$-consistent estimator $\hsigma$, 
because under Assumption~\ref{as:betasignal}, $\frn\|\bfX_{(S_2)}\beta_{0(S_2)}\|^2=o_p(n^{-1/2})$
does not affect the convergence rate of $\hsigma$. 
One can also consider a model selection procedure with a sure screening property \citep{FanLv08}, that is,
$\calG_{A_0} \subset \hat{M}$ with probability tending to one. \cite{Fan12} have shown that
$\hsigma$ constructed by the two-stage approach can be $\sqn$-consistent if 
$|\hat{M}|(\log p)/n=o_p(1)$. The authors also propose a refitted cross-validation estimator of $\sigma^2$
which only requires  $|\hat{M}|=o_p(n)$.
The lasso estimator satisfies the sure screening property 
under a suitable beta-min condition and may be used as the model selection procedure in the above two methods
for variance estimation. 
It is also possible to use $\hsigma$ with a different convergence rate, such as
the estimators in the scaled lasso \citep{SunZhang12} and the scaled group lasso \citep{MitraZhang14}.
See these references for the exact convergence rate of $\hsigma$, which may impose
a different scaling among $(n,J,q_0,p_{\max})$ for \eqref{eq:asymcondhsigma} to hold.

\subsection{Comparison to other methods}\label{sec:theorycomparison}

We discuss the main assumptions and conclusions of our asymptotic results,
in comparison with other competing methods.

Assumption~\ref{as:design}, imposed on the design matrix $\bfX$, is quite mild and holds for random Gaussian designs:
\begin{proposition}\label{prop:GaussianC1}
Assume each row of $\bfX$ is drawn independently from $\dnorm_p(0,\bmSigma)$ 
with covariance matrix $\bmSigma=\bmSigma(n)$.
Then Assumption \ref{as:design} holds if $(p_{\max}\vee\log J)/n\to 0$ and $\Lambda_k(\bmSigma)\in[c_*,c^*]$, with $0<c_*< c^*<\infty$, for all $k\in\N_p$.
\end{proposition}

\noindent
The random Gaussian design, often referred to as the $\bmSigma$-{\em Gaussian ensemble}, 
is a very common model used in high-dimensional inference.
Particularly, the de-biased lasso methods 
either assume the same model or use it to verify the regularity conditions on $\bfX$. 
The scaling among $(n,J,q_0,p_{\max})$ in \eqref{eq:sparsescaling} justifies 
the application of our method in a high-dimensional setting allowing $J\gg n \to \infty$ and $q_0,p_{\max}\to\infty$. 
If restricted to the special case of the lasso ($p_j=1$ for all $j$), 
\eqref{eq:sparsescaling} requires $s_0 \log p \ll \sqn$, 
which turns out to be the same scaling assumed in
the aforementioned de-biased lasso methods.
We have assumed in \eqref{eq:separablebeta} that the true coefficients of the active groups $A_0$ can
be separated into two subsets. The subset $S_1$ contains strong signals under
a beta-min condition on $\|\beta_{0(j)}\|$, 
only allowing its normalized $\ell_2$ norm to decay at a certain rate, while the subset $S_2$
includes small coefficients. This is a weaker assumption than 
the usual beta-min condition imposed on all nonzero components of $\beta_0$.
Suppose that $p_{\max}=\Omg(\log J)$ and $s_0\asymp n^{a/2}$ for $a\in[0,1)$.
Then the normalized $\ell_2$ norm of $\beta_{0(j)}$ may be of any order outside the 
interval $[n^{-(1+a)/2},n^{-(1-a)/2}]$. When $a$ is sufficiently small, we almost
remove the beta-min condition in the sense that the normalized $\ell_2$ norm of an active group 
can be of any order except $n^{-1/2}$. 
To apply Theorem~\ref{thm:asympknownvar} to the lasso, the beta-min condition on large signals becomes
$\inf_{S_1}|\beta_{0k}|\gg s_0(\log p /n)^{1/2}$. It is stronger than the minimum signal strength $(\log p /n)^{1/2}$
for variable selection consistency. This is because we did not assume any irrepresentable condition on $\bfX$
or require the point estimator $\tdbeta$ have an optimal convergence rate.
By Assumption~\ref{as:pluginest} it is sufficient for $\tdbeta$ to have the same suboptimal 
rate as the group lasso $\hbeta$ and to include only the strong signal groups. These are quite reasonable assumptions
satisfied by a simple thresholded group lasso.
Nevertheless, the signal strength assumption \eqref{eq:separablebeta} is relatively strong,
since an ideal inference method should be valid for any parameter value. It will be an important future 
contribution to develop bootstrap inference methods for high-dimensional data without such assumptions.

The de-biased lasso estimator $\hat{b}$, 
constructed with a relaxed inverse $\hTheta$ of the Gram matrix $\bfC$, is asymptotically unbiased and
can be expressed as 
\begin{equation}\label{eq:dsest}
\sqn(\hat{b}-\beta_0)=W + \Delta,
\end{equation} 
where $W$ is a Gaussian random vector and the bias term 
$\|\Delta\|_{\infty}=o_p(1)$. 
This result may not be directly applicable to group inference about $\beta_{0G}$ when $|G|\to\infty$.
To perform group inference, \cite{MitraZhang14} propose to
de-bias the group lasso with a relaxed projection matrix $P_G$ for each group $G$.
The authors establish that finding $P_G$ is feasible for certain sub-Gaussian designs 
with high probability, leaving the possibility of failing to find a suitable $P_G$ when $n$ is finite. 
These methods do not rely on any beta-min condition, 
which is an advantage over our approach. However, the importance of our theoretical results
are seen as follows. First, our results answer the fundamental question about
quantifying the uncertainty in the group lasso $\hbeta$, instead of a particular modification or function
of $\hbeta$ such as the de-biased estimators. In this sense, the two approaches
are not directly comparable. Second, \eqref{eq:DeltatoinP0} is a stronger result
that bounds the total $\ell_2$ deviation over all groups.
The uniform convergence in Corollary~\ref{thm:knownvarlaw},
not established for the de-biased group lasso,
provides the theoretical foundation for simultaneous inference on a large number of groups.

\myrev{
\cite{Chatterjee13} establish in their Theorem 5.1 the consistency of a residual bootstrap for the adaptive lasso \citep{Zou06}, which can be defined via \eqref{eq:grplassodef} with $p_j=1$ and $w_j$ specified by an initial
estimator $\beta^{\dag}$. Their beta-min condition, $\inf_{A_0}|\beta_{0j}|>K$ for some $K\in(0,\infty)$,
is obviously stronger than our Assumption~\ref{as:betasignal}. 
They require $\beta^{\dag}$ to be $\sqn$-consistent and to satisfy some form of deviation
bound, which are much more restrictive than our Assumption~\ref{as:pluginest} on the initial point estimator.
There are also a number of technical assumptions on the design matrix in \cite{Chatterjee13}.
It is not clear whether the random Gaussian design in Proposition~\ref{prop:GaussianC1} satisfies these assumptions.
On the other hand, due to the use of the adaptive lasso, their confidence intervals
will have lengths of the optimal rate $n^{-1/2}$, while our method applied to the lasso will construct
wider intervals asymptotically, as discussed in Remark~\ref{rm:anorderbound}.
}

\section{Numerical results}\label{sec:numerical}

\subsection{Methods and simulated data}\label{sec:simulationsettings}

To evaluate its finite-sample performance, we applied our parametric bootstrap method
on simulated and real data sets. 
For each data set, we obtain a solution path of the group lasso 
using the \texttt{R} package \texttt{grpreg} \citep{BrehenyHuang15}
and choose the tuning parameter $\lmd$ by cross-validation.
Let $\hbeta\in B(y;\hlmd)$ denote the solution for the chosen $\hlmd$ and $q=|G(\hbeta)|$
be the number of active groups of $\hbeta$. In light of \eqref{eq:orderofbth},
we set the threshold value $b_{\supth}=\frac{1}{2} \hlmd (q p_{\max})^{1/2}$ to obtain $\tdbeta$. 
When $\tdbeta$ has $n$ or more nonzero components,
we keep only its largest $\lfloor n/p_{\max}\rfloor-1$ groups in terms of $\ell_2$ norm. 
Then the active coefficients of $\tdbeta$ are re-computed via least-squares, and
the noise variance is estimated by the residual. Given $(\tdbeta,\hsigma)$,
we draw $N=300$ bootstrap samples of $\sbeta\in B(y^*;\hlmd)$ 
to make inference following the procedure described in Section~\ref{sec:pbinference}.

We compare our method with competitors, including 
two methods of the de-biased lasso approach \citep{Javanmard13b,vandeGeer13} 
and the group-bound method \citep{Meinshausen15},
implemented in the \texttt{R} package \texttt{hdi} \citep{Dezeure15} and \texttt{R} function \texttt{SSlasso}.
To distinguish from the de-biased lasso method of \cite{Javanmard13b}, we will
call the method of \cite{vandeGeer13} the de-sparsified lasso. 
The \texttt{hdi} package allows the user to input an estimate of the noise variance
for the de-sparsified lasso, for which we use the
same estimate in our approach to make results more comparable.
Other tuning parameters are chosen via the default methods in their respective implementation.

The rows of $\bfX$ are independent draws from
$\dnorm_p(0,\bmSigma)$, with $\bmSigma$ chosen from the following
two designs: (i) Toeplitz, $\bmSigma_{jk}=0.5^{|j-k|}$; (ii) Exponential decay, $(\bmSigma^{-1})_{jk}=0.4^{|j-k|}$.
Recall that $s_0$ denotes the number of active coefficients. 
We adopt two distinct ways to assign active coefficients: (1) Set the first $s_0$ coefficients,
$\beta_{0k}$, $k=1,\ldots,s_0$, to be nonzero; (2) the active coefficients are evenly spaced
in $\N_p$. Since neighboring $X_j$'s are highly correlated in both designs, the two different ways of assigning
active coefficients lead to distinct correlation patterns among the true predictors and between the true and false
predictors. Index the two designs by $d\in\{\text{i, ii}\}$ 
and the two ways of assigning active coefficients by $a\in\{1,2\}$.
Given $\bfX$ and $\beta_0$, the response $y$ is simulated from $\dnorm_n(\bfX\beta_0,\bfI_n)$.
In the first simulation study described in Sections~\ref{sec:numericalgroup} and \ref{sec:individual},
we fixed $s_0=10$ and drew $\beta_{0k} \sim \text{Unif}(-1,1)$ for $k\in \calG_{A_0}$.
We chose  $(n,p) \in \{(100,200),(100,400)\}$.
The combination of above choices created eight different data generation settings. 
In each setting, we generated $K=20$ data sets, i.e. $K$ independent realizations of $(y,\bfX,\beta_0)$.

\subsection{Group inference}\label{sec:numericalgroup}

We first examine the performance in group inference. The predictors were partitioned into groups
of size $p_j=10$ by two different methods. In the first method, we group the $10$ active
coefficients into one group and the other zero coefficients into the remaining groups, in which case
there is only one active group. In the second way of grouping, there are two active groups,
each containing five nonzero coefficients and five zero coefficients. We will denote these two ways of grouping
by $\calP_1$ and $\calP_2$. Clearly, the signal strength of the active groups in $\calP_2$ is weaker.

Our bootstrap method, the de-sparsified lasso, and the group-bound method 
were used to test the hypothesis $H_{0,j}:\beta_{0(j)}=0$ for each group.
The de-sparsified lasso method outputs a p-value 
$\xi_k$ for each individual test $\beta_{0k}=0$ for $k\in \N_p$. If $k\in \calG_j$, we adjust the p-value 
by Bonferroni correction
with the group size to obtain $\xi_{\text{adj},k}=\xi_k p_j$. Then the hypothesis $H_{0,j}$
will be rejected at level $\siglevel$ if $\min_{k\in\calG_j} \xi_{\text{adj},k} \leq \siglevel$.
For each $j\in\N_J$, the group-bound method constructs a lower bound
for $\|\beta_{0(j)}\|_1$ to test the hypothesis $H_{0,j}$ at level $\gamma$.
We chose $\siglevel=0.05$ and recorded the numbers of rejections among the active and the zero groups,
denoted by $m_A$ and $m_I$, respectively. Then for each method, we calculated the power
$\text{PWR}=m_A/q_0$ and the type-I error rate, i.e. false positive rate, $\text{FPR}=m_I/(J-q_0)$.
Our method can build a confidence region for each group \eqref{eq:regionest}, and we recorded
the coverage rate $r_A$ for the active groups. Note that the coverage rate for the zero groups $r_I=1-\text{FPR}$.
The other two competing methods do not construct confidence regions for a group of coefficients.
The average result over the $K$ data sets in each data generation setting is
reported in Table~\ref{tab:groupsim}. 

\begin{table}[t]
\caption{Coverage, power and false positive rate (\%) in group inference\label{tab:groupsim}}
\begin{center}
 \begin{tabular}{@{\extracolsep{4pt}}cccrrrrrrr@{}}
  \hline  \hline
  \multicolumn{3}{c}{Data Setting} & \multicolumn{3}{c}{bootstrap}  &  \multicolumn{2}{c}{de-sparsified} 
  & \multicolumn{2}{c}{group-bound}\\ 
  \cline{1-3}\cline{4-6}\cline{7-8}\cline{9-10}
 $(n,p)$ & $(a,d)$ & $\mathcal{G}$ & $r_A$ & PWR  & FPR &  PWR  & FPR &  PWR  & FPR \\   \hline
\multirow{8}{*}{$(100,200)$}		
		&\multirow{2}{*}{(1, i)} 	& $\calP_1$ & 95.0 	& 95.0 	& 5.5 & 100.0 	& 44.2 & 0.0 & 0.0 \\
		&					& $\calP_2$ & 92.5 	& 67.5 	& 5.3 & 97.5 	& 48.9 & 0.0 & 0.0 \\ 
		&\multirow{2}{*}{(1, ii)}	& $\calP_1$ & 95.0 	& 100.0	& 5.0 & 100.0 	& 50.3 & 0.0 & 0.0 \\ 
		&					& $\calP_2$ & 90.0 	& 97.5 	& 3.6 & 100.0 	& 54.4 & 0.0 & 0.0 \\ 
		&\multirow{2}{*}{(2, i)} 	& $\calP_1$ & 100.0	& 100.0 	& 5.3 & 100.0 	& 53.4 & 0.0 & 0.0 \\
		&					& $\calP_2$ & 90.0 	& 85.0 	& 4.7 & 100.0 	& 52.5 & 0.0 & 0.0 \\
		&\multirow{2}{*}{(2, ii)}	& $\calP_1$ & 100.0	& 100.0 	& 4.2 & 100.0 	& 61.8 & 0.0 & 0.0 \\
		&					& $\calP_2$ & 100.0	& 95.0 	& 4.7 & 100.0 	& 61.7 & 0.0 & 0.0 \\
 \hline
\multirow{8}{*}{$(100,400)$}		
		&\multirow{2}{*}{(1, i)} 	& $\calP_1$ & 95.0 	& 100.0 	& 4.9 & 100.0 & 42.6 & 0.0 & 0.0 \\
		&					& $\calP_2$ & 85.0 	& 75.0 	& 3.2 & 100.0 & 45.4 & 0.0 & 0.0 \\ 
		&\multirow{2}{*}{(1, ii)}	& $\calP_1$ & 90.0 	& 100.0 	& 4.4 & 100.0 & 64.2 & 0.0 & 0.0 \\ 
		&					& $\calP_2$ & 85.0 	& 87.5 	& 6.2 & 97.5 & 61.4 & 0.0 & 0.0 \\ 
		&\multirow{2}{*}{(2, i)} 	& $\calP_1$ & 90.0 	& 100.0 	& 4.0 & 100.0 & 65.8 & 0.0 & 0.0 \\
		&					& $\calP_2$ & 85.0 	& 67.5 	& 4.3 & 100.0 & 64.7 & 0.0 & 0.0 \\
		&\multirow{2}{*}{(2, ii)}	& $\calP_1$ & 100.0 	& 100.0 	& 4.0 & 100.0 & 74.6 & 0.0 & 0.0 \\
		&					& $\calP_2$ & 82.5	 	& 90.0 	& 3.0 & 100.0 & 66.6 & 0.0 & 0.0 \\
 \hline
\end{tabular}
\end{center}
{$r_A$: coverage rate of active groups; PWR: power; FPR: false positive rate.}	
\end{table}

The big picture of this comparison is very clear. Our bootstrap method shows a very satisfactory control of 
type-I errors, around the nominal level of $5\%$ for all cases, 
while its coverage rate for active groups is $>0.9$ with power $>0.8$ for a strong majority. 
In contrast, the de-sparsified lasso method is too optimistic 
with very high type-I error rates, ranging between $40\%$ and $70\%$,
and the group-bound approach is extremely conservative, resulting in no false rejections but having little power at all.
Although the bias term $\Delta$ \eqref{eq:dsest} can be far from negligible 
when $n$ is finite, the de-sparsified lasso method totally ignores this term, 
and as a result its confidence intervals are often too narrow and its p-values become severely underestimated.
On the contrary, our approach takes care of the bias in the group lasso via
simulation instead of asymptotic approximation, 
which turns out to be very important for finite samples as suggested by the comparison.
This is one of the reasons for the observed better performance of our method in Table~\ref{tab:groupsim}.
Another reason is our explicit use of the group lasso so that group structures are utilized in
both the estimation of $\tdbeta$ and the bootstrap simulation.
These points will be further confirmed in our subsequent comparison on inference for individual coefficients.

The group-bound method is by nature a conservative approach, testing the null hypothesis $\beta_{0G}=0$ for $G\subset \N_p$ with a lower-bound of the $\ell_1$
norm $\|\beta_{0G}\|_1$. By design, the type-I error is controlled simultaneously for all groups $G\subset\N_p$ at the significance level $\gamma$ even if one specifies
a particular group, like what we did in this comparison. It suffers from low power, 
especially when the group $G$ does not
include all the covariates that are highly correlated with the true variables. 
To verify this observation, we did more test on the data sets of size $(n,p)=(100,200)$. 
For the Toeplitz design 
with the first $10$ coefficients being active, its power stayed close to zero until we included the first 100
variables in the group $G$ and increased to 0.86 when $G=\{1,\ldots,p\}$ including all variables. 
We then increased the signal strength
by simulating active coefficients $\beta_{0k}\sim\text{Unif}(-3,3)$. In this case, the power of the group-bound method
increased to 0.45 for $G=\{1,\ldots,10\}$ and to 0.52 for $G=\{1,\ldots,50\}$, which are still substantially lower than the power of our bootstrap method under weaker signals; see the first row in Table~\ref{tab:groupsim}.
This numerical comparison demonstrates the advantage of our method in presence of between-group correlations, while the group-bound method might be more appropriate when groups are defined by
clustering highly correlated variables together. 
Since its target application is different, we exclude
the group-bound method from the following comparisons. 


\subsection{Individual inference}\label{sec:individual}

Since the de-sparsified lasso was designed without considering variable grouping,
we conducted another set of comparisons on inference about individual coefficients.
We included the de-biased lasso in these comparisons as well.
For our bootstrap method, we completely ignore any group structure and set $p_j=1$ for all $j$
throughout all the steps in our implementation. Under this setting, the confidence region 
$R_j(\siglevel)$ in \eqref{eq:regionest} reduces to an interval.
We applied the three methods on the same data sets used in the previous comparison
to construct $95\%$ confidence intervals for individual coefficients and to test
$H_k: \beta_{0k}=0$ for $k\in \N_p$.
To save space, the detailed results are relegated to Supplemental Material.
Here, we briefly summarize the key findings.

First, our bootstrap method
shows an almost perfect control of the type-I error, slightly lower than but very close to $5\%$,
while its power is quite high, between $0.5$ and $0.9$ for almost every setting.
The type-I error rate of the de-sparsified lasso is again substantially higher than the desired level.
The de-biased lasso, on the other hand, is seen to be conservative with
type-I error rate close to or below 1\% for all the cases. 
Second, in terms of interval estimation, our method shows a
close to $95\%$ coverage rate, averaging over all coefficients, with a shorter interval length
than the other two competitors.
Our bootstrap method built shorter intervals with coverage rate close to $95\%$ for zero coefficients,
while the other two methods showed higher coverage rates with slightly wider intervals for active coefficients.
We observe that the interval lengths between active and zero coefficients are very different for our method
but are almost identical for the two de-biased lasso methods. 
For $j\notin A_0$, the variance of the lasso $\hbeta_j$ is in general smaller
and $\hbeta_j$ can be exactly zero. Our method makes use of such sparsity
to improve the efficiency of interval estimation for zero coefficients. The 
de-biased lasso $\hat{b}$ \eqref{eq:dsest} de-sparsifies all components
of the lasso, in some sense averaging the uncertainty over all coefficients.
Lastly, the effect of grouping variables can be seen by comparing the results for our parametric bootstrap 
with those in Table~\ref{tab:groupsim}: At almost the same level of type-I error rate, its power and coverage rate of 
active coefficients can be boosted substantially through either way of grouping,
which numerically confirms our motivation to group
coefficients for a more sensitive detection of signals. 

Finally, we compared the running time between our method applied to the lasso and the de-sparsified lasso method.
The time complexity of our method is in proportion to the bootstrap sample size $N$,
which has been fixed to $N=300$ above. 
We ran both methods without parallelization to make inference about all $p$ coefficients. 
For a wide range of choices for $(n,p)$, our method was uniformly,
in some settings more than 10 times, faster than the de-sparsified lasso.
Bootstrapping the group lasso
is even faster, since the group lasso in general runs faster than the lasso.

\subsection{Weak and dense signals}\label{sec:weakgroup}

In the second simulation study, we added a third active group of 10 coefficients
to the setting of $(n,p)=(100,400)$ with grouping $\calP_2$. We set the coefficients $\beta_{0k}\in\{\pm\eps\}$
in the third group and chose $\eps\in\{0.02,0.2\}$. The signal of this group, in particular 
when $\eps=0.02$, was much weaker than that of the first two groups. 
The vector $\beta_0$ also became denser with $q_0=3$ and $s_0=20$. 
Both aspects make the data sets more challenging for an inferential method.
We note that neither the sparsity condition \eqref{eq:sparsescaling} nor 
the signal strength condition \eqref{eq:separablebeta} in Assumption~\ref{as:betasignal} is satisfied.
The results here thus can indicate how the bootstrap method works when key
assumptions of our asymptotic theory are violated.
To obtain accurate estimation of coverage rates in presence of such small signals,
we increased the number of data sets generated in each setting to $K=50$.

We applied our bootstrap method and the de-sparsified lasso with Bonferroni adjustment
to perform group inference on these data sets, as we did in Section~\ref{sec:numericalgroup}.
Moreover, we conducted a Wald test with the de-sparsified lasso $\hat{b}$ as follows.
The asymptotic distribution of $\hat{b}$ \eqref{eq:dsest} implies that, for a fixed group $G$ of size $m$,
$T_G=n(\hat{b}_G-\beta_{0G})^\trans V_{GG}^{-1}(\hat{b}_G-\beta_{0G})$ follows
a $\chi^2$ distribution with $m$ degrees of freedom as $n\to\infty$, where $V$ is the covariance
of the Gaussian random vector $W$. Thus, one may use the statistic $T_G$ to test 
whether a group of coefficients is zero such as in $H_{0,j}:\beta_{0(j)}=0$.
The results of the three methods are reported in Table~\ref{tab:weakgroup}. 
To highlight the expected difference in performance, we separately report 
the coverage rate and power
for the strong groups, $r_{S}$ and PWR$_S$, and those for the weak group, $r_{W}$ and PWR$_W$.

\begin{table}[ht]
\caption{Coverage rate, power and false positive rate in presence of weak and dense signals \label{tab:weakgroup}}
\begin{center}
 \begin{tabular}{@{\extracolsep{4pt}}ccrrrrrrrr@{}}
  \hline  
\multicolumn{2}{c}{ } & \multicolumn{4}{c}{$\epsilon=0.02$} & \multicolumn{4}{c}{$\epsilon=0.2$} \\ \cline{3-6} \cline{7-10}
 & {$(a,d)$} & (1, i) & (1, ii) & (2, i) & (2, ii) & (1, i) & (1, ii) & (2, i) & (2, ii) \\ \hline
bootstrap 			& $r_{S}$			& 86.0 & 86.0 & 83.0 & 84.0 & 86.0 & 87.0 & 82.0 & 87.0 \\  
				& $r_{W}$			& 42.0 & 60.0 & 52.0 & 62.0 & 46.0 & 12.0 & 14.0 & 22.0 \\  
				& PWR$_S$		& 77.0 & 94.0 & 75.0 & 94.0 & 66.0 & 94.0 & 68.0 & 87.0 \\  
				& PWR$_W$		& 8.0 & 6.0 & 4.0 & 6.0 & 84.0 & 54.0 & 70.0 & 78.0 \\  
				& FPR			& 4.9 & 4.2 & 5.0 & 4.1 & 2.6 & 3.6 & 3.2 & 2.2 \\  \hline
de-sparsified 		& PWR$_S$		& 99.0 & 99.0 & 100.0 & 100.0 & 97.0 & 99.0 & 98.0 & 97.0 \\ 
(Bonferroni)		& PWR$_W$		& 18.0 & 32.0 & 8.0 & 20.0 & 84.0 & 48.0 & 70.0 & 78.0 \\  
				& FPR			& 20.8 & 23.9 & 16.2 & 18.6 & 18.2 & 27.5 & 16.9 & 19.1 \\  \hline
de-sparsified 		& PWR$_S$		& 71.0 & 88.0 & 83.0 & 91.0 & 58.0 & 85.0 & 70.0 & 81.0 \\  
(Wald)			& PWR$_W$		& 0.0 & 0.0 & 0.0 & 0.0 & 30.0 & 0.0 & 0.0 & 0.2 \\  
				& FPR			& 0.0 & 0.2 & 0.1 & 0.1 & 0.0 & 0.0 & 0.0 & 0.3 \\  \hline
\end{tabular}
\end{center}
$r_S$, $r_W$: coverage rate for strong and weak signal groups, respectively;
PWR$_S$, PWR$_W$: power for strong and weak signal groups, respectively;
FPR: false positive rate.
\end{table}

Again our bootstrap method achieved a good control of type-I error rate, all around
or below 5\%. The coverage rate and power for the strong groups are comparable to those 
in Table~\ref{tab:groupsim}, indicating that they were not affected by the inclusion of a weak group.
The power for detecting the third group is low when $\eps=0.02$, which is fully expected given
the low signal strength, and becomes much higher when $\eps$ is increased to $0.2$.
As discussed above, the data simulated here do not satisfy Assumption~\ref{as:betasignal}.
As a consequence, the bootstrap samples might not provide
a good approximation to the sampling distribution, which could be a reason
for the low coverage rate of the weak group.
We observe that $r_W$ was in general higher when $\epsilon=0.02$ than when $\epsilon=0.2$.
This is because $\sup_{S_2}\|\beta_{0(j)}\|$ with $\epsilon=0.02$ is closer to the requirement in
\eqref{eq:separablebeta} for small coefficients.
The de-sparsified lasso with Bonferroni correction failed to control the type-I error rate
at the desired level. While it shows a higher detection power for the weak group
when $\eps=0.02$,  its power is largely comparable to our method when $\eps=0.2$.
The gain in power could simply be the result of high false positive rates.
Compared to Table~\ref{tab:groupsim}, we see a decrease in the false positive rates. 
This is because the error variance $\sigma^2$ was overestimated for these data sets, 
which alleviated the underestimate of p-values by the de-sparsified lasso.
On the contrary, the Wald test seems too conservative, almost never rejecting any
zero group. Its power of detecting the weak group is close to zero for most of the cases.
These results are the consequence of ignoring the term $\Delta$ in \eqref{eq:dsest},
which introduces systematic bias in the Wald test statistic for finite samples.
This numerical comparison demonstrates the advantage of our bootstrap method
in the existence of weak coefficient groups under a relatively dense setting.

\subsection{Real data designs}\label{sec:realdata}

We further tested our method on design matrices drawn from a gene expression data set
\citep{ivanova2006}, which contains expression profiles for about $40,000$ mouse genes across $n=70$ samples.
The expression profile of each gene was transformed to a standard normal distribution via quantile 
transformation. The following procedure was used to generate data sets for our comparison.
First, randomly pick $p$ genes and denote their expression profiles by $X_j\in\R^n$ for $j=1,\ldots,p$. 
We calculate the correlation coefficients $(r_{ij})_{p\times p}$ among $X_j$'s and
the total absolute correlation $r_{i\bullet}=\sum_j |r_{ij}|$ for each gene $i\in\N_p$.
For the gene with the highest $r_{i\bullet}$, we identify
the $(m-1)$ genes that have the highest absolute correlation with this gene. These $m$ genes
are put into one group of size $m$. Then we remove them from the gene set and
repeat this grouping process until we partition all $p$ genes into $J=p/m$ groups. 
This grouping mechanism results in high correlation among covariates in the same group.
Next, fixing the first $q_0$ groups to be active,
we draw their coefficients $\beta_{0k}\sim \text{Unif}(-b,b)$. 
The parameters in the above procedure were chosen as $p\in\{500,1000\}$, $b\in\{1,3,5\}$,
$m=10$, and $q_0=3$. For each combination of $(p,b)$, we obtained $K=100$ independent
realizations of $(\bfX,\beta_0)$. Given each realization, a range of the noise variance 
$\sigma^2\in\{0.1,0.5,1\}$ was then used to simulate the response $y\sim \dnorm_n(\bfX\beta_0,\sigma^2\bfI_n)$.
Compared to the data generation settings in Section~\ref{sec:simulationsettings}, data sets in this subsection 
have a smaller sample size $n=70$
but a higher dimension $p$, and the correlation among the covariates is much higher.  
These put great challenges on an inferential method.

\begin{table}[t]
\caption{Comparison of power and false positive rate on gene expression data\label{tab:realdata}}
\begin{center}
 \begin{tabular}{@{\extracolsep{4pt}}ccrrrrrrrr@{}}
  \hline  \hline
\multicolumn{3}{c}{\multirow{2}{*}{Data Setting}} & \multicolumn{2}{c}{Group inference} & 
\multicolumn{5}{c}{Individual inference ($p_j=1$)} \\
\cline{6-10}
    & & & \multicolumn{2}{c}{bootstrap} & \multicolumn{2}{c}{bootstrap}  &  \multicolumn{3}{c}{de-sparsified lasso} \\ 
  \cline{1-3}\cline{4-5}\cline{6-7}\cline{8-10}
 $p$ & $\beta_0$ & $\sigma^2$ & PWR  & FPR & PWR  & FPR &  PWR  & FPR & PWR$^*$  \\   \hline
  \multirow{9}{*}{$500$}
   			& \multirow{3}{*}{$(-1,1)$} 		&	0.1		& 50.3 & 1.6 & 14.2 & 3.1 & 41.9 & 19.3 & 12.9 \\ 
			&						& 	0.5		& 24.7 & 1.8 & 13.1 & 3.1 & 36.1 & 15.0 & 12.0 \\ 
			& 						&	1		& 24.7 & 2.4 & 12.0 & 3.0 & 31.8 & 12.7 & 13.0 \\ 
   			& \multirow{3}{*}{$(-3,3)$} 		&	0.1		& 61.3 & 1.0 & 15.7 & 3.3 & 51.0 & 27.1 & 13.2 \\  
			&						& 	0.5		& 54.7 & 1.1 & 15.0 & 3.2 & 46.2 & 21.0 & 10.7 \\ 
			& 						&	1		& 48.7 & 1.2 & 15.1 & 3.4 & 47.3 & 22.4 & 11.2 \\ 
   			& \multirow{3}{*}{$(-5,5)$} 		&	0.1		& 58.7 & 1.4 & 14.9 & 3.1 & 48.5 & 24.4 & 14.3 \\ 
			&						& 	0.5		& 57.7 & 1.2 & 14.9 & 3.1 & 43.1 & 20.0 & 10.8 \\ 
			& 						&	1		& 55.7 & 0.9 & 14.2 & 3.1 & 44.4 & 19.2 & 14.3 \\ 
 \hline
   \multirow{9}{*}{$1000$}
   			& \multirow{3}{*}{$(-1,1)$} 		&	0.1		& 41.0 & 1.1 & 12.4 & 2.0 & 36.3 & 14.0 & 12.0 \\ 
			&						& 	0.5		& 29.3 & 2.0 & 11.5 & 2.0 & 29.8 & 10.6 & 15.3 \\
			& 						&	1		& 30.3 & 1.1 & 10.4 & 1.8 & 29.4 & 10.5 & 13.2 \\ 
   			& \multirow{3}{*}{$(-3,3)$} 		&	0.1		& 41.0 & 0.8 & 13.9 & 2.1 & 34.8 & 13.9 & 14.2 \\
			&						& 	0.5		& 33.7 & 0.7 & 12.2 & 2.1 & 37.4 & 15.5 & 13.8 \\  
			& 						&	1		& 46.7 & 1.3 & 12.8 & 2.1 & 39.5 & 17.7 & 13.4 \\  
   			& \multirow{3}{*}{$(-5,5)$} 		&	0.1		& 50.0 & 1.0 & 12.4 & 2.1 & 33.6 & 12.3 & 13.9 \\ 
			&						& 	0.5		& 47.7 & 0.9 & 12.8 & 2.1 & 34.3 & 14.3 & 10.9 \\
			& 						&	1		& 45.7 & 0.8 & 11.8 & 2.1 & 36.9 & 15.9 & 11.9 \\  
 \hline
 \end{tabular}
 \end{center}
FPR: false positive rate; PWR: power; PWR$^*$: power of the de-sparsified lasso after matching false positive rate
to 5\%.
\end{table}

We applied both the bootstrap and the de-sparsified lasso methods to perform group inference as we did
in Section~\ref{sec:numericalgroup}. As reported in Table~\ref{tab:realdata}, our bootstrap method
gives a good and slightly conservative control over type-I errors, 
with false positive rates all close to but below $5\%$, the desired level. Its power in general increases
as the signal-to-noise ratio increases and 
is seen to be around 0.5 when the signal-to-noise ratio is reasonably high. 
On the contrary, the type-I error rate of the de-sparsified lasso method, not reported in the table, 
was even $>0.9$ for most of
the cases, showing that it failed to provide an acceptable p-value approximation for these data sets.
This might be caused by the facts that this method is not designed for group inference and that $n=70$ is
too small for asymptotic approximation. To conduct a complete comparison, we then used
both methods to make inference about individual coefficients as in Section~\ref{sec:individual},
in which the group structures were totally ignored in our method by setting all $p_j=1$.
Our method again controlled the type-I error to a level slightly lower than $5\%$, but
showed a decrease in power, as expected, without utilizing grouping. 
The false positive rate of the de-sparsified lasso became smaller for individual inference, 
ranging between $15\%$ and $30\%$, but
still far from the desired level of $5\%$. This makes it difficult to compare
power between the two methods, as the observed higher power of one method could
simply come at the cost of more false positives. To resolve this issue, we sorted 
the p-values for all the zero coefficients output from the de-sparsified lasso method,
and chose a cutoff $p^*$ such that $5\%$ of them would be rejected. In this way, the false positive rate by definition
is always $5\%$, slightly higher than that of our bootstrap method, 
while the corresponding power
becomes largely comparable. We also noticed that the overestimate of
the significance level by the de-sparsified lasso method was severe for these data sets: 
To achieve the target type-I error rate of 
$5\%$, the cutoff for its p-values turned out to be $<0.002$ for all the settings and 
was even much smaller for many of them. 

This comparison shows that our parametric bootstrap method can achieve a desired 
level of false positive control in presence of high correlation among a large number of predictors.
It again confirms that grouping variables can lead to substantial power gain.

\subsection{Sensitivity to thresholding}\label{sec:sensitivity}

With $\hlmd$ chosen by cross-validation, the only parameter that requires user input in our method
is the threshold value $b_{\supth}$. 
The following experiment has been conducted to examine how sensitive our method is to this parameter. 
Given the group lasso solution $\hbeta$, we reorder its groups so that
$\|\hbeta_{(1)}\|\geq \ldots \geq \|\hbeta_{(J)}\|$.
Then we choose a range of threshold values, $b_{\supth}=\|\hbeta_{(k+1)}\|$,
such that the thresholded $\tdbeta$ has $k$ active groups for $k=0,\ldots,K$, say $K=6$.
We applied this procedure on the simulated data sets generated 
from the four settings with $n=100,p=400$ and $\calG=\calP_2$ in Table~\ref{tab:groupsim}, 
which have $q_0=2$ active groups. These were the most difficult settings for which
our bootstrap method had the lowest power and coverage rate $r_A$.
Figure~\ref{fig:sensitivity} plots the curve of the false positive rate and 
the curve of the active group coverage rate $r_A$ against $k$, the number of active groups after thresholding,
in each of the four settings.

\begin{figure}[t]
\centering
 \includegraphics[height=0.95\linewidth,angle=-90,trim=0.5in 0in 0in 0in,clip]{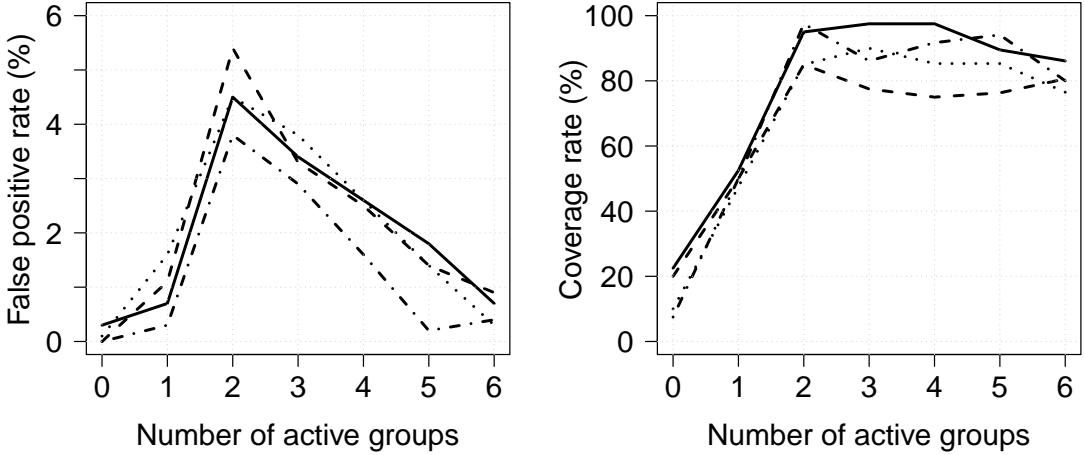} \\
   \caption{Sensitivity analysis on thresholding:
   (left) false positive rate and (right) coverage rate of active groups against the number of
   active groups of thresholded group lasso for the four settings
   of $(a,d)=$ (1, i) (solid), (1, ii) (dash), (2, i) (dot), and (2, ii) (dot-dash).   \label{fig:sensitivity}}
\end{figure}

The false positive rates are well-controlled at the desired level of $5\%$ for all
the threshold values. They are around 0.05 when $b_{\supth}$ is well-chosen so that $\tdbeta$ 
has $k=q_0=2$ active groups, and become smaller when  $b_{\supth}$ deviates from the optimal value.
This suggests that our method is not sensitive to the threshold value in terms of type-I error control.
The coverage of the active groups stays at a high level when $\tdbeta$ contains two or more active groups,
but can be substantially lower if one or both of the true active groups are missing. 
Thus, including a few zero groups in the active set of $\tdbeta$ will not hurt the performance  of our method 
that much, since via refitted least-squares, the estimated coefficients of these groups tend to be small.
Similar patterns were observed for inference on individual coefficients, when $b_{\supth}$ was chosen
for $\tdbeta$ to have up to 30 nonzero coefficients while the true active set 
contained only $10$ variables.

\section{Generalizations and discussions}\label{sec:generalizations}

\subsection{Block lasso and sub-Gaussian error}\label{sec:blocksubG}

The bootstrap method outlined in Section~\ref{sec:pbinference} can be generalized to other 
sparse regularization methods under different error distributions in an obvious way.
However, the difficulty is to validate such generalizations theoretically. In this subsection,
we provide theory for two generalizations:
First, we assume that the error vector $\varepsilon \in \R^n$ is zero-mean sub-Gaussian. 
Second, we consider the general $(1,\alp)$-group norm \eqref{eq:groupnorm} for $\alpha\in [2,\infty]$ 
and correspondingly the block lasso estimator
\begin{equation}\label{eq:blocklassodef}
\hbeta \in \argmin_{\beta} \frac{1}{2} \|y-\bfX \beta \|^2 + n \lambda \| \beta\|_{\calG,\alpha}.
\end{equation}
Our method is essentially to bootstrap the block lasso under a sub-Gaussian error distribution. We assume
the error distribution is given from which we can draw $\varepsilon^*$.
Define $\hdelta$, $\shdelta$, and $\sdelta$ as in \eqref{eq:defalldelta} but with $\hbeta$,
$\hbeta^*$ and $\sbeta$ denoting the corresponding block lasso estimates instead.

Recall that $\alp^*$ is conjugate to $\alp$.
To establish asymptotic theory for bootstrapping the block lasso, we need a modified version of 
Assumption~\ref{as:betasignal}:

\begin{assumption}\label{as:betasignal_alp}
The true coefficient vector $\beta_0$ is sparse:
\begin{equation}\label{eq:sparsescaling_alp}
{p_{\max}^{2/\alp^*-1} q_0 (p_{\max} \vee \log J)}/{\sqn} =o(1),
\end{equation}
and $A_0=S_1 \cup S_2$ such that
\begin{equation}\label{eq:condbeta0_alp}
\inf_{j\in S_1}\frac{\|\beta_{0(j)}\|_{\alp}}{\sqpj} \gg \frac{b_n(\alp)}{\sqn}
 \quad\text{and}\quad \sup_{j\in S_2}\frac{\|\beta_{0(j)}\|_{\alp}}{\sqpj} \ll \frac{1}{b_n(\alp)\sqn},
\end{equation}
where $b_n(\alp)=p_{\max}^{1/\alp^*-1/2} s_0\{1\vee{(\log J/p_{\max}})^{1/2}\}$.
\end{assumption}

Together with a proper choice of $\lambda$, 
we can now generalize Theorem~\ref{thm:asympknownvar} to the block lasso.
\begin{theorem}\label{thm:consistency_alpha}
Consider the model \eqref{eq:model} with sub-Gaussian noise $\varepsilon$
and $p_{\max}/p_{\min}\asymp 1$. Let $\alpha \in [2,\infty)$ and $\frac{1}{\alpha}+\frac{1}{\alpha^*}=1$.
Suppose that Assumptions \ref{as:design}, \ref{as:pluginest} and \ref{as:betasignal_alp} hold.
Consider a suitable choice of 
\begin{equation}\label{eq:lambdaasy_alp}
\lambda \asymp p_{\max}^{1/\alp^*-1/2}\{(p_{\max} \vee \log J)/{n}\}^{1/2}.
\end{equation} 
Let $r_n$ be as in \eqref{eq:defrn} for some $\hbeta$ defined by \eqref{eq:blocklassodef}
and $\tddelta=r_n(b-\beta_0)$, where $b$ is any minimizer of \eqref{eq:blocklassodef}.
Suppose that $\nu[\varepsilon^*]=\nu[\varepsilon]$ in the definition of $\shdelta$ and $\sdelta$. 
Then for every $\epsilon>0$,
\begin{align}
&\Prob(\|\hdelta-\tddelta\| > \epsilon \mid \bfX)=o_p(1), \label{eq:difftwohdelta_alp} \\
&\Prob(\|\sdelta-\shdelta\| > \epsilon \mid \bfX,\tdbeta)=o_p(1). \label{eq:DeltatoinP0_alp}
\end{align}
\end{theorem}

The assumptions of this theorem are in parallel to those of Theorem~\ref{thm:asympknownvar}. 
The differences appear in Assumption~\ref{as:betasignal_alp} on $\beta_0$
and the order of $\lambda$ \eqref{eq:lambdaasy_alp},
both reducing to the corresponding assumptions in Theorem~\ref{thm:asympknownvar} when $\alp=\alp^*=2$.
This theorem does not include the case $\alp=\infty$. For this case, we need to impose an additional
assumption on the margin of $\beta_{0}$ defined as follows. 
For $\theta=(\theta_1,\ldots,\theta_m)\in \R^m$, let $\pi$ be a permutation of the set $\N_m$
such that $|\theta_{\pi(1)}|\geq \ldots \geq |\theta_{\pi(m)}|$. Define the margin of $\theta$ by
\begin{equation}\label{eq:margin}
d(\theta)= \frac{1}{\surd{2}}(|\theta_{\pi(1)}|-|\theta_{\pi(2)}|).
\end{equation}

\begin{theorem}\label{thm:consistency_inf}
Let $\alp=\infty$ and $\alp^*=1$. In addition to the assumptions of Theorem~\ref{thm:consistency_alpha},
further assume that
\begin{equation}\label{eq:marginasy}
\inf_{j\in S_1} d(\beta_{0(j)}) \gg \lambda (q_0)^{1/2}.
\end{equation}
Then all the conclusions in Theorem~\ref{thm:consistency_alpha} hold.
\end{theorem}

The additional assumption on the margin of $\beta_{0(j)}$ ensures that the $\ell_\infty$ norm
is differentiable in a neighborhood of $\beta_{0(j)}$.
Letting $\alp^*=1$ in \eqref{eq:sparsescaling_alp} and \eqref{eq:lambdaasy_alp}, we have
\begin{align*}
\lambda (q_0)^{1/2}\asymp \left\{\frac{p_{\max}q_0 (p_{\max} \vee \log J)}{n}\right\}^{1/2}
\ll \frac{p_{\max}q_0 (p_{\max} \vee \log J)}{\sqn}=o(1).
\end{align*}
It is seen that assumption \eqref{eq:marginasy} is quite mild, allowing the margin of $\beta_{0(j)}$
to decay to zero.

Let $a_n$ be defined by 
$a_n \sup_{j} \|\hbeta_{(j)}-\beta_{0(j)}\|/\sqpj \asymp_P 1$
for the block lasso $\hbeta$. We can obtain similar result as that in Corollary~\ref{thm:knownvarlaw} 
for $\alp\in[2,\infty]$, although it is unclear when $a_n$ would become optimal 
in this more general case for statistical inference.
We have assumed that the error distribution is known, $\nu[\varepsilon^*]=\nu[\varepsilon]$, in the above.
This may be relaxed to using an estimated error distribution. 
If we assume that the only unknown parameter of the sub-Gaussian error distribution
is a scale parameter, one can show that a point estimator satisfying certain convergence rate
will suffice for establishing the above theorems. 
More general situations are to be studied in the future.

\begin{remark}\label{rm:generalnorm}
In the course of proving Theorems~\ref{thm:consistency_alpha} and \ref{thm:consistency_inf},
we derived nonasymptotic upper bounds on $\|\hbeta-\beta_0\|$ and $\|\hbeta-\beta_0\|_{\calG,\alpha}$ for
$\alp\in[2,\infty]$. This result is of independent interest and obtained under 
a weaker block normalization assumption compared to 
Corollary 4 in \cite{Negahban12}. 
See Supplemental Material for details.
\end{remark}

\subsection{Future work}\label{sec:future}

We have developed asymptotic theory on the consistency of
a parametric bootstrap method for group norm penalized estimators, which
allows for the use of simulation to construct interval estimates and quantify estimation 
uncertainty under group sparsity. 
Due to the intrinsic bias of a sparse penalized estimator, however,
the length of an estimated interval, in general, may not be on the order of the optimal parametric rate of $n^{-1/2}$;
see, for example, \eqref{eq:defan}. 
One possible improvement is to simulate from a less biased estimator instead, such as
the de-biased estimator in \cite{ZhangZhang11} and \cite{vandeGeer13}, 
as discussed in Section~\ref{sec:theorycomparison}.
In order to reach the optimal rate,
the authors rely on solving $p$ lasso problems to obtain the relaxed inverse $\hTheta$, which
becomes a computational bottleneck for this approach.
We may use a different relaxed inverse $\hTheta$ that is computationally cheaper
to define the de-biased estimator $\hat{b}=\hat{b}(\hTheta,y,\bfX)$. 
The same bootstrap method can
be applied to approximate the distribution, $\nu[r_n(\hat{b}-\beta_0)\mid\bfX]$, with $r_n$
determined by the convergence rate of $\hat{b}$. If $\hat{b}$ converges
at a faster rate as it is less biased than the group lasso $\hbeta$, the confidence intervals will 
be shorter asymptotically. Using the bootstrap instead of asymptotic approximation, this approach 
is also expected to have superior finite-sample performance, as was observed in the
numerical comparisons in Section~\ref{sec:numerical}. 
In a similar spirit, \cite{NingLiu14} used a bootstrap strategy to approximate
the distribution of their decorrelated score function for high-dimensional inference.

We have proposed multiple approaches that can provide a point estimate $\tdbeta$ for our bootstrap method.
However, it is arguable that thresholding the group lasso is still the most convenient choice in practice,
without solving another optimization problem. From this perspective, an interesting
future direction is to develop a method to determine an appropriate threshold value from data.
It remains to find out whether Assumption~\ref{as:pluginest} on $\tdbeta$ is a necessary condition
by further analysis. At least, the numerical results in Section~\ref{sec:sensitivity} seem to suggest
that this may not be the case.
Another future direction is to develop 
estimator augmentation \citep{Zhou14} under group sparsity, which, by employing Monte Carlo
methods, could offer great flexibility in sampling from
the distribution of a sparse regularized estimator.


\myappendix
\renewcommand{\thefigure}{S\arabic{figure}}
\setcounter{figure}{0}
\renewcommand{\thetable}{S\arabic{table}}
\setcounter{table}{0}
\renewcommand{\theassumption}{S\arabic{assumption}}
\setcounter{assumption}{0}
\renewcommand{\theremark}{S\arabic{remark}}
\setcounter{remark}{0}

\section{Proofs of results in Section~\ref{sec:theory}}\label{sec:normality}

\subsection{Preliminaries and published results}\label{sec:prepub}

Throughout Sections~\ref{sec:normality} and \ref{sec:unknownvar}, 
we assume that $\varepsilon  \sim \dnorm_n(0,\sigma^2\bfI_n)$ 
and $w_j\in[w_*,w^*]$ for all $j$ with the positive constants $w_*<w^*<\infty$.
Let $\bfC=\frn \bfX^{\trans} \bfX$ be the Gram matrix,
$U = \frn \bfX^{\trans} \varepsilon \in \R^p$, and
$\bfW=\diag(w_1 \bfI_{p_1},\ldots,w_J \bfI_{p_J})\in\R^{p\times p}$.
To motivate the definitions of the centered and rescaled estimators in \eqref{eq:defalldelta},
we note that the penalized loss in \eqref{eq:grplassodef}, up to an additive constant, is identical to
\begin{align}\label{eq:defVn}
 &\, \frac{1}{2} \|y-\bfX \beta \|^2-\frac{1}{2} \|y-\bfX \beta_0 \|^2 + 
 n \lambda \sum_{j=1}^J w_j (\| \beta_{(j)}\| -\| \beta_{0(j)}\|)  \nonumber\\
= &\, \frac{n}{2r_n^2} \delta^{\trans} \bfC \delta - \frac{n}{r_n}\delta^{\trans} U 
+ n\lambda \sum_{j \in A_0} w_j \left(\| \beta_{0(j)}+ r_n^{-1}\delta_{(j)}\|- \| \beta_{0(j)}\|\right) 
+ \frac{n\lambda}{r_n} \sum_{j \notin A_0} w_j \|\delta_{(j)} \| \nonumber\\
\defi &\, V(\delta; \beta_0, U), 
\end{align}
where $\delta = r_n (\beta - \beta_0)\in\R^p$.
Put $U^*=\bfX^{\trans} \varepsilon^*/n$. It follows from \eqref{eq:defalldelta} that
\begin{align}
\hdelta & = r_n(\hbeta-\beta_0) \in \argmin_{\delta} V(\delta;\beta_0,U) \label{eq:defhdelta} \\
\shdelta&=r_n(\hbeta^*-\beta_0) \in\argmin_\delta V(\delta;\beta_0,U^*) \label{eq:defshdelta}\\
\sdelta&= r_n(\sbeta-\tdbeta) \in \argmin_{\delta} V(\delta;\tdbeta,U^*). \label{eq:defsdelta}
\end{align}
We wish to approximate the distribution $\nu[\hdelta]$ by the conditional distribution $\nu[\delta^*\mid \tdbeta]$.

We first collect relevant existing results on the group lasso.
The most relevant are the upper bounds of the $\ell_{1,2}$ and the $\ell_2$ errors
of the group lasso, under the restricted eigenvalue assumption.
For $A\subset \N_J$, define the cone
\begin{equation}\label{eq:conedef}
\scrC(A) = \left\{\Delta\in\R^p:
\sum_{j\in A^c} w_j \|\Delta_{(j)}\|  \leq 3 \sum_{j\in A} w_j \| \Delta_{(j)}\|\right\}.
\end{equation}

\begin{assumption}[RE($m$)]
For $m \in \N_{J}$,
\begin{equation}\label{eq:REdef}
\kappa(m) \defi \min_{|A| \leq m} \min_{\Delta \ne 0} 
\left\{\frac{\|\bfX \Delta\|}{\sqn \|\Delta_{(A)} \|}: \Delta \in \scrC(A) \right\} >0.
\end{equation}
\end{assumption}

This assumption is used by \cite{Lounici11} to derive error bounds for the 
group lasso, which generalizes the restricted eigenvalue assumption for $\ell_1$ regularization \citep{Bickel09}.
The restricted eigenvalue assumption, or the closely related
compatibility condition \citep{vandeGeer09}, is one of the weakest on the design matrix for obtaining
useful results for the lasso and its variates.
We state the following result on relevant error bounds for the 
group lasso from Theorem 3.1 and its proof in \cite{Lounici11}.
Let $w_{\min}$ and $w_{\max}$ denote, respectively, the minimum and the maximum of $\{w_j\}$.

\begin{theorem}\label{thm:grouplasso}
Consider the model \eqref{eq:model} with $\varepsilon \sim \dnorm_n(0,\sigma^2\bfI_n)$ and let $J\geq 2$, $n \geq 1$. 
Suppose $|G(\beta_0)|\leq q$ and Assumption RE$(q)$ is satisfied. 
Then on the event $\calE=\cap_{j=1}^J\{\|U_{(j)}\|\leq w_j \lambda/2\}$,
for any solution $\hbeta$ of \eqref{eq:grplassodef} we have
\begin{equation}\label{eq:boundG2norm}
\|\hbeta-\beta_0\|_{\calG,2}  \leq \frac{16\lambda}{\kappa^2(q)} \sum_{j\in A_0}\frac{w_j^2}{w_{\min}}\defi h_1. \end{equation}
If in addition Assumption RE$(2q)$ holds, then on the same event,
\begin{align}
\|\hbeta-\beta_0 \| & \leq \frac{4\surd{10}}{\kappa^2(2q)} \frac{\lambda \sum_{A_0}w_j^2}{w_{\min}\surd{q}}\defi \tau. 
\label{eq:boundL2}
\end{align}
Choose $a>1$ and 
\begin{align}\label{eq:condlambda}
\lambda \geq \frac{2\sigma}{\sqn} \max_{j\in \N_J} \frac{1}{w_j}
\left[\textup{tr}(\bfC_{(jj)})+2 \Lambda_{\max}(\bfC_{(jj)})\left\{2a\log J+ (a p_j  \log J)^{1/2}\right\}\right]^{1/2}.
\end{align}
Then $\Prob(\calE)\geq 1-2J^{1-a}$.
\end{theorem}

Another key result is the following inequality from Proposition 1 in \cite{Negahban12},
which has been used to establish the so-called restricted strong convexity condition
for least-squares loss under group norm regularization.

\begin{lemma}\label{lm:NegahbanProp1}
Let $\alpha\in[1,\infty]$, $Z\sim\dnorm_p(0,\bfI_p)$, and 
\begin{equation}\label{eq:rhostar}
\rho(\alpha^*)=\E\left\{\max_{j\in \N_J} \|Z_{(j)}\|_{\alpha^*}/\sqn\right\}, 
\end{equation}
where $\alpha^*$ is conjugate to $\alpha$ satisfying $\frac{1}{\alpha}+\frac{1}{\alpha^*}=1$. 
Assume that each row of $\bfX$ is drawn independently from $\dnorm_p(0,\bmSigma)$ with $\bmSigma>0$. 
Then there are positive constants $(\kappa_1,\kappa_2)$ that depend only on $\bmSigma$ 
such that, with probability greater than $1-c_1\exp(-c_2 n)$,
\begin{equation}\label{eq:groupstrconvex}
\frac{1}{n}\|\bfX\Delta\|^2 \geq \kappa_1 \|\Delta\|^2 - \kappa_2 \rho^2(\alpha^*) \|\Delta\|^2_{\calG,\alpha}
\quad\text{for all } \Delta\in\R^p.
\end{equation}
\end{lemma}

We have assumed \eqref{eq:groupstrconvex} in Assumption~\ref{as:design}. The above
lemma shows that this assumption holds with high probability
for the random Gaussian design. It is likely that \eqref{eq:groupstrconvex} also holds with high probability
for sub-Gaussian designs, 
as suggested by the analysis of \cite{Rudelson13}. 
Moreover, the restricted eigenvalue assumption is implied by
inequality \eqref{eq:groupstrconvex} when $n$ is large: 

\begin{lemma}\label{lm:GaussianRE}
Assume \eqref{eq:groupstrconvex} holds. Then Assumption RE$(q)$ holds 
with $\kappa(q)\geq (\kappa_1/2)^{1/2}$ if $n>c_3\kappa_3q(p_{\max}\vee \log J)$, 
where $\kappa_3=\kappa_2/\kappa_1$ and $c_3$ depends only on $w_{\max}/w_{\min}$.
\end{lemma}

\subsection{Proof overview}\label{sec:overview}

A key step in the proof of Theorem~\ref{thm:asympknownvar} is to establish the nonasymptotic bound
for $\|\sdelta-\shdelta\|$ contained in Theorem~\ref{thm:finiteknownvar}, 
from which the main asymptotic results follow.
Before going through the details, we briefly overview the basic ideas behind the proof. 

Let $\bfX$ be a fixed design matrix 
such that \eqref{eq:groupstrconvex} holds. Put $\Delta=\sdelta-\shdelta$.
Our goal is then to find an upper bound, say $B_n$, such that
$\Prob(\|\Delta\|^2\leq B_n \mid \tdbeta)$ is close to one for a large set of $\tdbeta$; 
see \eqref{eq:finiteboundDelta} in Theorem~\ref{thm:finiteknownvar}.
Rewrite \eqref{eq:groupstrconvex} for $\alp=\alp^*=2$ as
\begin{align*}
\|\Delta\|^2 \leq \frac{1}{\kappa_1}\left\{\frac{1}{n}\|\bfX\Delta\|^2+ \kappa_2 \rho^2(2)\|\Delta\|_{\calG,2}^2 \right\},
\end{align*}
which shows that we need to control three terms, $\frn\|\bfX\Delta\|^2$, $\rho^2(2)$ and  $\|\Delta\|_{\calG,2}$.
Lemma~\ref{lm:rhostar} below provides an upper bound for $\rho(\alpha^*)$, which reduces to
\begin{equation}\label{eq:rho2}
\rho(2)\leq \left(5p_{\max}/n\right)^{1/2}+\left(4\log J /n\right)^{1/2}
\end{equation}
for $\alpha^*=2$. The bounds for the other two terms will be developed conditioning on
a few events. 
Let $\eta\in(0,1)$ and $M_2>0$ be constants. We define 
\begin{align}\label{eq:defA01}
A_{01}=A_{01}(\eta)= \left\{j\in A_0: \frac{\|\beta_{0(j)}\|}{\sqpj}> \frac{M_2}{r_n(1-\eta)}\right\}
\end{align}
and $A_{02}=A_0\setminus A_{01}$,
\begin{align}
\eta_1  = \sup_{j \in A_{01}} \frac{\|\tdbeta_{(j)}-\beta_{0(j)} \|}{\| \beta_{0(j)}\|},  \label{eq:difftdbeta}
\end{align}
and $\scrB_0\subset\R^p$ by
\begin{equation}\label{eq:defB0}
\scrB_0=\{\tdbeta\in\R^p:G(\tdbeta)\subset A_{01} \text{ and } \eta_1 \leq \eta\}.
\end{equation}
Consider events
\begin{align}
\calE^* &=\cap_{j=1}^J\{\|U^*_{(j)}\|\leq w_j \lambda/2\},  \label{eq:defcalE*}\\
E^*_{t} &=\{\shdelta \in \scrD(tM_2)\} \text{ for }t \in (0,1], \label{eq:defE*t}\\
E_2&=\{\sdelta\in \scrD(M_2)\}, \label{eq:defE2}
\end{align}
where $\scrD(c)$ is defined in \eqref{eq:defdeltaregion}.
For a fixed $\bfX$, the first two events are in the probability space for $U^*$ 
while $E_2$ is in the joint probability space for $U^*$ and $y$. 
By Theorem~\ref{thm:grouplasso}, $\|\Delta\|_{\calG,2}$ can be bounded using \eqref{eq:boundG2norm}
on the event $\{\tdbeta \in \scrB_0\}\cap \calE^*$. We will find an upper bound for $\frn\|\bfX\Delta\|^2$
in Lemma \ref{lm:boundDeltaforbeta} on the event $\{\tdbeta \in \scrB_0\}\cap E^*_1 \cap E_2$.
Putting together, for $\tdbeta \in \scrB_0$ we have
\begin{equation*}
\Prob(\|\Delta\|^2\leq B_n \mid \tdbeta) \geq \Prob(\calE^* \cap E^*_1 \cap E_2 \mid \tdbeta)
\geq \Prob(\calE^*\cap E^*_t \cap E_2  \mid \tdbeta)
\end{equation*}
for every $t\in(0,1)$. By Lemma~\ref{lm:eventsubset}, the last conditional probability 
reduces to $\Prob(\calE^*\cap E^*_t)$. This leads to the nonasymptotic bound in Theorem~\ref{thm:finiteknownvar}.
We further show that $B_n=o(1)$ and $\Prob(\calE^*\cap E^*_t)\to 1$ under our
asymptotic framework to establish the main results.

The upper bound for $\frn\|\bfX(\sdelta-\shdelta)\|^2$ is established by next two lemmas, of which
the first is a special case of Lemma~\ref{lm:boundVforbeta_alpha}. 
Define
\begin{equation}\label{eq:defrho}  
\eta_2  = \sup_{j\in A_{01}} \frac{M_2\sqpj}{r_n \|\beta_{0(j)}\|},
\end{equation}
where $M_2$ is the same constant as in \eqref{eq:defE*t} and \eqref{eq:defE2}.
By definition $\eta_1+\eta_2<1$ for $\tdbeta\in\scrB_0$.

\begin{lemma}\label{lm:boundVforbeta}
For any $\tdbeta\in\scrB_0$, $u\in \R^p$ and $\delta\in \scrD(M_2)$, we have
\begin{align}\label{eq:boundVforbeta}
& |V(\delta; \tdbeta, u) - V(\delta;\beta_0,u)| \nonumber\\
& \quad\quad\quad\quad\leq \frac{M_2 n \lambda (\eta_1+\eta_2)}{r_n(1-\eta_1-\eta_2)}
\sum_{A_{01}} w_j \sqpj I(p_j>1) + 2 n\lambda \sum_{A_{02}} w_j \|\beta_{0(j)}\| \defi h_2. 
\end{align}
\end{lemma}

\begin{lemma}\label{lm:boundDeltaforbeta}
Recall $\shdelta$ and $\sdelta$ defined in \eqref{eq:defshdelta} and \eqref{eq:defsdelta}. 
If $\tdbeta\in\scrB_0$ and $\sdelta,\shdelta\in \scrD(M_2)$, then
\begin{equation}\label{eq:boundDeltaforbeta}
\frac{1}{n}\|\bfX(\sdelta-\shdelta)\|^2 \leq 4 r_n^2 h_2/n.
\end{equation}
\end{lemma}

Suppose $\eta\in(0,1)$ is sufficiently small. The assumptions on $\beta_0$ and $\tdbeta$ 
can be expressed in a more explicit way to show that the estimate $\tdbeta\in\scrB_0$ is close to $\beta_0$.
The subset $A_{01}$ contains active groups with a large $\ell_2$ norm and 
$\|\tdbeta_{(j)}-\beta_{0(j)}\|\leq \eta \|\beta_{0(j)}\|$ for $j\in A_{01}$. For $j\notin A_{01}$,
by definition $\|\beta_{0(j)}\|$ is small or zero, and
$\tdbeta_{(j)}=0$ according to \eqref{eq:defB0}.

\begin{lemma}\label{lm:eventsubset}
Assume that $q_0 (p_{\max} \vee \log J) \ll \sqn$, $r_n=O(\sqn)$ and $r_n^2 h_2/n=o(p_{\min})$. 
Choose $\lambda=O(\{(p_{\max} \vee \log J)/n\}^{1/2})$.
Suppose that \eqref{eq:groupstrconvex} holds with universal constants $(\kappa_1,\kappa_2)$ when $n$ is large. 
Then for every $t\in(0,1)$, there is $N_1$ such that 
\begin{equation}\label{eq:inclusion}
\{\tdbeta \in \scrB_0\} \cap \calE^* \cap E^*_t \subset E_2
\end{equation}
when $n>N_1$.
\end{lemma}

As an immediate consequence of \eqref{eq:inclusion}, for any $t\in(0,1)$ and $\tdbeta\in\scrB_0$,
\begin{equation}\label{eq:PrJointEvent}
\Prob(\calE^*\cap E^*_t \cap E_2\mid \tdbeta) = \Prob(\calE^* \cap E^*_t\mid \tdbeta) = \Prob (\calE^* \cap E^*_t)
\end{equation}
when $n>N_1$. 
Now we are ready to establish a nonasymptotic bound for the deviation $\|\sdelta-\shdelta\|$, regarding $\bfX$
as a fixed matrix that satisfies a couple assumptions.

\begin{theorem}\label{thm:finiteknownvar}
Consider the model \eqref{eq:model} with $\varepsilon \sim \dnorm_n(0,\sigma^2\bfI_n)$, 
$n \geq 1$ and $J\geq 2$. Assume $|G(\beta_0)|\leq q$. 
Suppose that Assumption RE$(q)$ and \eqref{eq:groupstrconvex} hold. 
Define $\shdelta$ and $\sdelta$ by \eqref{eq:defshdelta} and \eqref{eq:defsdelta}, respectively, 
with $\varepsilon^* \sim \dnorm_n(0,\sigma^2\bfI_n)$. If \eqref{eq:inclusion} holds for some $t\in(0,1)$, 
then on the event that $\{\tdbeta\in\scrB_0\}$ we have
\begin{equation}\label{eq:finiteboundDelta}
\Prob\left\{\|\sdelta-\shdelta\|^2\leq \frac{4r_n^2}{\kappa_1 n}\left[h_2+\kappa_2n\rho^2(2)h_1^2\right] 
\left|\, \tdbeta\right.\right\}\geq \Prob(\calE^* \cap E^*_t),
\end{equation}
where $h_1$ and $h_2$ are defined in \eqref{eq:boundG2norm} and \eqref{eq:boundVforbeta}, respectively.
\end{theorem}

Our proof of Theorem~\ref{thm:asympknownvar} is based on Theorem~\ref{thm:finiteknownvar}. 
We will show that the assumptions
for \eqref{eq:finiteboundDelta} are satisfied with high probability
and then derive the desired conclusions from there. 
For the detailed proof, see Section~\ref{sec:proofthmknvar}.

\begin{remark}\label{rm:orderofupperbound}
In step 2 of the proof of Theorem~\ref{thm:asympknownvar},
the exact order of the upper bound in \eqref{eq:finiteboundDelta} is derived: 
Conditioning on a large set of $(\tdbeta,\bfX)$,
\begin{equation}\label{eq:exactorder}
\|\sdelta-\shdelta\|^2=O_p\left\{\frac{b_n}{\sqn}\sup_{S_1}\frac{\sqpj}{\|\beta_{0(j)}\|}
+b_n\sqn \sup_{S_2}\frac{\|\beta_{0(j)}\|}{\sqpj}
+\frac{q_0^2(p_{\max}\vee\log J)^2}{n}\right\},
\end{equation}
which under Assumption~\ref{as:betasignal} implies $\|\sdelta-\shdelta\|=o_p(1)$.
\myrev{Under orthogonal designs, \eqref{eq:groupstrconvex} holds with $\kappa_1=1$ and $\kappa_2=0$.
Then the upper bound can be improved to $\|\sdelta-\shdelta\|^2\leq 4r_n^2h_2/n$, and consequently,
the last term on the right side of \eqref{eq:exactorder} drops.}
\end{remark}

\subsection{Proofs of auxiliary results}\label{sec:apxproofs_sec2}

\begin{proof}[Proof of Lemma~\ref{lm:GaussianRE}]
For $\Delta \in \scrC(A)$ \eqref{eq:conedef},
\begin{equation*}
\|\Delta\|_{\calG,2} \leq (1+3w_{\max}/w_{\min}){|A|^{1/2}} \|\Delta_{(A)}\|
\end{equation*}
and therefore
\begin{equation*}
\frac{1}{n}\|\bfX\Delta\|^2 \geq 
\left[\kappa_1  - \kappa_2(1+3w_{\max}/w_{\min})^2 \rho^2(2)|A|\right] \|\Delta_{(A)}\|^2.
\end{equation*}
This implies 
\begin{equation*}
\kappa^2(q) \geq \kappa_1  - \kappa_2(1+3w_{\max}/w_{\min})^2 \rho^2(2)q,
\end{equation*}
where $\kappa(q)$ is defined in \eqref{eq:REdef}.
With \eqref{eq:rho2} we can bound $\kappa(q)$
from below, $\kappa(q)\geq (\kappa_1/2)^{1/2}$, as long as $n>c_3\kappa_3 q(p_{\max}\vee \log J)$, 
where $\kappa_3=\kappa_2/\kappa_1$ and $c_3$ is a constant that only depends on $w_{\max}/w_{\min}$.
\end{proof}

We prove an inequality that will be used in the proofs of a few results.

\begin{lemma}\label{lm:convexity}
Let $\gamma$ be any minimizer of $V(\delta; \tdbeta, u)$ for $\tdbeta\in\R^p$ and $u\in\R^p$. 
Then for any $\Delta \in \R^p$,
\begin{equation}\label{eq:convexlowbound}
V(\gamma+\Delta; \tdbeta, u)-V(\gamma; \tdbeta, u)\geq \frac{n}{2r_n^2} \Delta^{\trans} \bfC \Delta.
\end{equation}
\end{lemma}
\begin{proof}
Let $b=r_n^{-1}\gamma + \tdbeta$. Direct calculation gives
\begin{align*}
& V(\gamma+\Delta; \tdbeta, u)-V(\gamma; \tdbeta, u) \\
& \quad =  \frac{n}{2r_n^2} \Delta^{\trans} \bfC (\Delta + 2 \gamma)- \frac{n}{r_n}\Delta^{\trans} u 
 + n\lambda \sum_{j=1}^J w_j \left(\| b_{(j)}+ r_n^{-1}\Delta_{(j)}\|- \| b_{(j)}\|\right).
\end{align*}
The KKT condition for $\gamma$ to minimize $V(\delta;\tdbeta,u)$ is
\begin{align}\label{eq:KKTminV}
r_n^{-1}\bfC\gamma + \lambda \bfW s -u =0,
\end{align}
where $s\in\R^p$ and $s_{(j)}$ is a subgradient of $\|\beta_{(j)}\|$ at $b_{(j)}$. 
By the definition of a subgradient,
\begin{align}\label{eq:subgradineq}
\| b_{(j)}+ r_n^{-1}\Delta_{(j)}\|-\|b_{(j)}\| \geq r_n^{-1} \Delta_{(j)}^{\trans} s_{(j)}
\end{align}
for all $j=1,\ldots,J$. Now we have
\begin{align*}
& V(\gamma+\Delta; \tdbeta, u)-V(\gamma; \tdbeta, u) \\
& \quad \geq \frac{n}{2r_n^2} \Delta^{\trans} \bfC \Delta + \frac{n}{r_n^2} \Delta^{\trans} \bfC \gamma
- \frac{n}{r_n}\Delta^{\trans} u +\frac{n\lambda}{r_n} \sum_{j=1}^J w_j \Delta_{(j)}^{\trans} s_{(j)}  \\
& \quad = \frac{n}{2r_n^2} \Delta^{\trans} \bfC \Delta + \frac{n}{r_n} \Delta^{\trans} 
\left(r_n^{-1}\bfC\gamma - u + \lambda \bfW s \right) = \frac{n}{2r_n^2} \Delta^{\trans} \bfC \Delta,
\end{align*}
where we have used \eqref{eq:subgradineq} and \eqref{eq:KKTminV}.
\end{proof}

\begin{proof}[Proof of Lemma~\ref{lm:boundVforbeta}]
This follows as a special case of Lemma~\ref{lm:boundVforbeta_alpha}.
\end{proof}

\begin{proof}[Proof of Lemma~\ref{lm:boundDeltaforbeta}]
By Lemma~\ref{lm:boundVforbeta},
\begin{align*}
|V(\delta;\tdbeta,U^*)-V(\delta;\beta_0,U^*)| &\leq h_2, \text{ for } \delta\in\{\sdelta,\shdelta\}.
\end{align*}
Let $\Delta=\shdelta-\sdelta$. We have
\begin{align}\label{eq:diffofV}
V(\shdelta;\beta_0,U^*) & \geq V(\shdelta;\tdbeta,U^*) - h_2 \nonumber\\
& \geq V(\sdelta;\tdbeta,U^*) + \frac{n}{2r_n^2} \Delta^{\trans} \bfC \Delta -h_2 \nonumber\\
& \geq V(\sdelta;\beta_0,U^*)+ \frac{n}{2r_n^2} \Delta^{\trans} \bfC \Delta -2h_2, 
\end{align}
where the second inequality is due to Lemma~\ref{lm:convexity}. 
Since $V(\sdelta;\beta_0,U^*)\geq V(\shdelta;\beta_0,U^*)$ by definition \eqref{eq:defshdelta}, 
the desired inequality follows. 
\end{proof}

\begin{proof}[Proof of Lemma~\ref{lm:eventsubset}]
We prove by contradiction. 
Condition on $\{\tdbeta \in \scrB_0\} \cap \calE^* \cap E^*_t$, and suppose $\sdelta\notin \scrD(M_2)$,
that is, $(E_2)^c$ happens. 
For $x\in\R^p$, define function 
\begin{equation*}
f(x)= \sup_{j\in A_0}\|x_{(j)}\|/\sqpj,
\end{equation*} 
which is subadditive and continuous. 
By our hypothesis, $f(\sdelta)>M_2$. Thus,
there is $c \in (0,1)$ such that $f(c\shdelta+ (1-c)\sdelta)=M_2$, since $f(\shdelta)\leq tM_2<M_2$. 
Let $\gamma=c\shdelta + (1-c)\sdelta$. By \eqref{eq:defsdelta},
\begin{equation}\label{eq:Vgamma1}
V(\gamma; \tdbeta,U^*) \geq V(\sdelta; \tdbeta,U^*).
\end{equation}
Since $\tdbeta \in \scrB_0$, \eqref{eq:boundVforbeta} applies to both $\gamma$ and $\shdelta$. 
Using a similar argument as in \eqref{eq:diffofV}, we can show that 
\begin{equation*}
V(\gamma;\tdbeta,U^*) \geq V(\shdelta;\tdbeta,U^*)+ \left(\frac{n}{2r_n^2} \Delta^{\trans} \bfC \Delta -2h_2\right),
\end{equation*}
where $\Delta=\gamma-\shdelta=(1-c)(\sdelta-\shdelta)$. If $\Delta^{\trans} \bfC \Delta > 4r_n^2h_2/n$, then
\begin{equation}\label{eq:Vgamma2}
V(\gamma; \tdbeta,U^*) > V(\shdelta; \tdbeta,U^*).
\end{equation}
Combining \eqref{eq:Vgamma1} and \eqref{eq:Vgamma2},
\begin{align*}
V(\gamma; \tdbeta,U^*)  > c V(\shdelta; \tdbeta,U^*) + (1-c) V(\sdelta; \tdbeta,U^*)
 \geq V(\gamma; \tdbeta,U^*),
\end{align*}
where the second inequality is due to the convexity of $V$. This leads to a contradiction. Therefore, it remains to show $\Delta^{\trans} \bfC \Delta > 4r_n^2h_2/n$ when $n$ is greater than some $N_1$. According to \eqref{eq:groupstrconvex}, a sufficient condition for this is
\begin{equation}\label{eq:liminfsuff}
\liminf_{n\to\infty} \left\{\kappa_1 \|\Delta\|^2 - \kappa_2 \rho^2(2) \|\Delta\|^2_{\calG,2} -4r_n^2h_2/n\right\} >0.
\end{equation}

The first term
\begin{align*}
\kappa_1 \|\Delta\|^2 & \geq \kappa_1 p_{\min}[f(\gamma-\shdelta)]^2 \\
&\geq \kappa_1 p_{\min}[f(\gamma)-f(\shdelta)]^2\geq \kappa_1 (1-t)^2 M_2^2p_{\min},
\end{align*} 
where the subadditivity of $f$ is used for the second inequality.
The third term $4r_n^2h_2/n=o(p_{\min})$ by assumption. By Lemma~\ref{lm:GaussianRE}, \eqref{eq:groupstrconvex} implies RE$(q_0)$ when $n>c_3\kappa_3q_0(p_{\max}\vee \log J)=o(\sqn)$ 
and $\liminf_n\kappa(q_0)>0$. Since $|G(\tdbeta)|\leq |A_0|=q_0$ for $\tdbeta\in\scrB_0$, by \eqref{eq:boundG2norm} in Theorem~\ref{thm:grouplasso}, on event $\calE^*$ we have $\|\sdelta\|_{\calG,2}\leq r_n h_1$ and $\|\shdelta\|_{\calG,2}\leq r_n h_1$ with $q=q_0$, which implies
\begin{equation}\label{eq:G2norm}
\|\Delta\|_{\calG,2}\leq \|\sdelta\|_{\calG,2}+\|\shdelta\|_{\calG,2} \leq 2 r_n h_1.
\end{equation}
Since $\liminf_n\kappa(q_0)>0$, the second term
\begin{align}\label{eq:order2ndterm}
\kappa_2 \rho^2(2) \|\Delta\|^2_{\calG,2}  = O(r_n^2 \rho^2(2) \lambda^2 q_0^2) 
 =O(q_0^2(p_{\max}\vee\log J)^2/n)=o(1),
\end{align}
where we have used the facts that $r_n=O(\sqn)$ and 
both $\rho^2(2)$ and $\lambda^2$ are $O({(p_{\max}\vee\log J)/n})$. 
Therefore, \eqref{eq:liminfsuff} holds, and a contradiction is reached when $n$ is large.
\end{proof}

\begin{proof}[Proof of Theorem~\ref{thm:finiteknownvar}]
For $\tdbeta\in\scrB_0$, $|G(\tdbeta)|\leq |G(\beta_0)|\leq q$ and thus, 
by Theorem~\ref{thm:grouplasso}, $\|\sdelta\|_{\calG,2}\leq r_n h_1$ 
and $\|\shdelta\|_{\calG,2}\leq r_n h_1$ on event $\calE^*$. Let $\Delta=\sdelta-\shdelta$. 
Then, on the event $\calE^*\cap \{\tdbeta\in\scrB_0\}$, we have
$\|\Delta\|_{\calG,2}\leq 2 r_n h_1$ as in \eqref{eq:G2norm}.
By Lemma~\ref{lm:boundDeltaforbeta}, $\{\tdbeta\in\scrB_0\}\cap E_1^* \cap E_2$ implies
$\frac{1}{n}\|\bfX\Delta\|^2 \leq 4 r_n^2 h_2/n$.
Putting together with \eqref{eq:groupstrconvex}, 
conditioning on $\{\tdbeta\in\scrB_0\}\cap \calE^* \cap E_1^* \cap E_2$,
\begin{align*}
\|\Delta\|^2 &\leq \frac{1}{\kappa_1}\left\{\frac{1}{n}\|\bfX\Delta\|^2+ \kappa_2 \rho^2(2)\|\Delta\|_{\calG,2}^2 \right\} \\
& \leq \frac{4r_n^2}{\kappa_1 n}\left\{h_2+\kappa_2n\rho^2(2)h_1^2\right\}.
\end{align*}
Thus, for $\tdbeta\in\scrB_0$ and $t\in(0,1)$, we have
\begin{align*}
\Prob\left[\|\Delta\|^2\leq \frac{4r_n^2}{\kappa_1 n} \left\{h_2+\kappa_2n\rho^2(2)h_1^2\right\} \mid \tdbeta\right] & \geq \Prob(\calE^* \cap E^*_1 \cap E_2 \mid \tdbeta) \\
& \geq \Prob(\calE^* \cap E^*_t \cap E_2 \mid \tdbeta) \\
& = \Prob (\calE^* \cap E^*_t),
\end{align*}
where \eqref{eq:PrJointEvent}, a consequence of \eqref{eq:inclusion}, has been used to obtain the last equality.
\end{proof}

\subsection{Proofs of main results}
\label{sec:proofthmknvar}

\begin{proof}[Proof of Theorem~\ref{thm:asympknownvar}]
The proof is broken down into four steps.

Step 1: We verify all the assumptions of 
Theorems~\ref{thm:grouplasso} and \ref{thm:finiteknownvar} hold with high probability.
Since $q_0 (p_{\max} \vee \log J)\ll\sqn$, 
by Lemma~\ref{lm:GaussianRE} and Assumption~\ref{as:design}, 
RE$(c_3q_0)$ is satisfied with high probability for any $c_3>0$
and $\liminf_n\kappa(c_3 q_0)> 0$. 
As a consequence of Assumption~\ref{as:design}, $\max_j\textup{tr}(\bfC_{(jj)})=O_p(p_{\max})$, and 
thus a suitable choice of $\lambda \asymp \{(p_{\max} \vee \log J)/n\}^{1/2}$ satisfies \eqref{eq:condlambda}. The above arguments show that all the assumptions of Theorem~\ref{thm:grouplasso} are satisfied with high probability. Put $q=q_0$. We see that $\kappa(q_0)$ and $\kappa(2q_0)$ in \eqref{eq:boundG2norm} and \eqref{eq:boundL2} are bounded from below by a positive constant, and therefore $h_1=O(\lambda q_0)$ and $\tau=O(\lambda \surd{q_0})$. 
It follows from \eqref{eq:defrn} that
\begin{align*}
1 & \asymp_P (r_n/\sqpmax) \sup_{A_0} \|\hbeta_{(j)}-\beta_{0(j)}\| \\
 &\leq (r_n/\sqpmax) \|\hbeta-\beta_0\| =O_p\left\{r_n \lambda (q_0/p_{\max})^{1/2}\right\},
\end{align*} 
which implies $1/r_n=O(\lambda (q_0/p_{\max})^{1/2})$. Consequently, the order of $r_n$ is given by
\begin{equation}\label{eq:orderofr_n}
1/r_n=O\left\{\lambda (q_0/p_{\max})^{1/2}\right\}\quad\text{ and }\quad r_n=O(\sqn).
\end{equation}
Then, \eqref{eq:separablebeta} implies that, when $n$ is large, $S_1=A_{01}$ and $S_2=A_{02}$,
with $A_{01}$ and $A_{02}$ defined by \eqref{eq:defA01}.
We will show in Step 2 that $r_n^2h_2/n=o(1)$ and
invoke Lemma~\ref{lm:eventsubset} to obtain \eqref{eq:inclusion} when $n$ is large. 
Thus, all the assumptions of Theorem~\ref{thm:finiteknownvar} are satisfied with high probability. 

Step 2: We demonstrate that the upper bound for 
$\|\sdelta-\shdelta\|^2$ in \eqref{eq:finiteboundDelta} is $o(1)$.
It is immediate that
\begin{equation*}
r_n^2\rho^2(2) h_1^2=O(r_n^2\rho^2(2)\lambda^2q_0^2)=O(q_0^2(p_{\max}\vee\log J)^2/n)=o(1),
\end{equation*}
by \eqref{eq:orderofr_n}, \eqref{eq:rho2} and \eqref{eq:sparsescaling}.
To bound $h_2$ \eqref{eq:boundVforbeta}, we first show that $\eta_1=O(\eta_2)=o(1)$. 
Recall the respective definitions of $\scrB(M_1)$, $\eta_1$ and $\eta_2$ 
in \eqref{eq:BM2}, \eqref{eq:difftdbeta} and \eqref{eq:defrho}. 
For $\tdbeta\in\scrB(M_1)$,
\begin{equation}\label{eq:eta1and2}
\eta_1=\sup_{j\in A_{01}}\frac{\|\tdbeta_{(j)}-\beta_{0(j)} \|}{\| \beta_{0(j)}\|} 
 \leq \sup_{j\in A_{01}}\frac{M_1\sqpj}{r_n\| \beta_{0(j)}\|} =(M_1/M_2) \eta_2.
\end{equation}
On the other hand, by \eqref{eq:separablebeta} and \eqref{eq:orderofr_n},
\begin{align}
\eta_2  =\frac{1}{r_n} \sup_{A_{01}}\frac{M_2\sqpj}{\| \beta_{0(j)}\|}
\ll \frac{1}{r_n q_0\lmd(p_{\max})^{1/2}}=O(q_0^{-1/2} p_{\max}^{-1}) \label{eq:upperrho} 
\end{align}
and thus $\eta_2=o(1)$.  
It then follows from \eqref{eq:boundVforbeta} that
\begin{equation}\label{eq:boundh2}
\frac{r_n^2h_2}{n}=O\left(r_n \eta_2  q_0 \lambda\sqpmax\right)
+O\left(r_n^2\lambda q_0 \sup_{S_2} \|\beta_{0(j)}\|\right)=o(1),
\end{equation}
by plugging \eqref{eq:upperrho} and \eqref{eq:separablebeta} in the first and the second term, respectively.

Step 3: We prove \eqref{eq:difftwohdelta}, which follows from \eqref{eq:finiteboundDelta}. 
Since $\sup_{A_0}\|\hdelta_{(j)}\|=O_p(\sqpmax)$ by \eqref{eq:defrn} 
and $p_{\max}\asymp p_{\min}$,
for every $\epsilon>0$ and $t\in(0,1)$, we can find $M_2<\infty$ and $N_1$ 
such that $\Prob(E_t) \geq 1-\frac{1}{2}\epsilon$ when $n\geq N_1$, 
where $E_t=\{\hdelta\in \scrD(tM_2)\}$. Define
\begin{equation}\label{eq:REX}
\scrX=\{\bfX\in\R^{n\times p}: 
\text{\eqref{eq:groupstrconvex} holds with universal constants $(\kappa_1,\kappa_2)$}\}.
\end{equation}
Since $\Prob(\bfX\in\scrX)\to 1$ by Assumption~\ref{as:design}, 
$\Prob(E_t\mid \bfX\in\scrX)\geq 1-\epsilon$ when $n\geq N_2$ for some $N_2$.
Recall the event $\calE$ in Theorem~\ref{thm:grouplasso}. 
If we choose $U^*=U$ and fix $\tdbeta=\beta_0$, then setting $h_2=0$ in \eqref{eq:finiteboundDelta} leads to
\begin{align*}
\Prob\left\{\|\hdelta-\tddelta\|^2\leq \left.(4\kappa_2/{\kappa_1})r_n^2\rho^2(2)h_1^2=o(1) \,\right| \bfX\in\scrX \right\} & \geq \Prob(\calE \cap E_t \mid \bfX\in\scrX) \\
& \geq 1-(2J^{1-a}+\epsilon),
\end{align*}
and hence \eqref{eq:difftwohdelta} holds. 

Step 4: To establish \eqref{eq:DeltatoinP0}, we first show $\Prob(\tdbeta\in\scrB_0)\to 1$
and $\Prob(\calE^* \cap E^*_t\mid \bfX\in\scrX)\to 1$.
Since $\eta_1\to0$ for $\tdbeta\in\scrB(M_1)$, there is $N_3$ such that 
$n\geq N_3$ implies $\scrB(M_1)\subset \scrB_0$ \eqref{eq:defB0} and thus
$\Prob(\tdbeta\in\scrB_0)\geq \Prob(\tdbeta\in\scrB(M_1)) \to 1$.
Note that $\sup_{A_0}\|\shdelta_{(j)}\|=O_p(\sqpmax)$ as $\nu[U^*]=\nu[U]$.
By the arguments in Step 3, $\Prob(\calE^*\mid\bfX\in\scrX )=\Prob(\calE\mid\bfX\in\scrX)\to 1$ 
and $P(E^*_t\mid\bfX\in\scrX)\to 1$. 
Thus, $\Prob(\calE^* \cap E^*_t\mid\bfX\in\scrX)\to 1$ as $n\to\infty$. 
Letting $n\to\infty$ in \eqref{eq:finiteboundDelta} gives
\begin{equation*}
\Prob\left\{\left.\|\sdelta-\shdelta\|^2=o(1) \,\right| \bfX,\tdbeta \right\} \to 1 \quad\text{in probability},
\end{equation*}
which then implies \eqref{eq:DeltatoinP0} 
and completes the proof.
\end{proof}

\begin{proof}[Proof of Corollary~\ref{thm:knownvarlaw}]
By assumption, there is a positive constant $c<\infty$ such that $\|L_j v\|\leq c\|v\|$ for all $j$.
It follows from \eqref{eq:groupwisebound} with $p_{\max}\asymp p_{\min}$ that
\begin{equation} 
a_n\sup_{j\in \N_J} \|\hbeta_{(j)}-\beta_{0(j)}\|/\sqpj =O_p(1). \nonumber
\end{equation}
By the way $r_n$ is defined in \eqref{eq:defrn}, we have $a_n=O(r_n)$, and thus
\eqref{eq:DeltatoinP0} of Theorem~\ref{thm:asympknownvar} implies
\begin{equation}\label{eq:anconverge}
\Prob\left\{\sup_{j\in \N_J}\left\|a_n (\hbeta^*_{(j)}-\beta_{0(j)})-a_n (\sbeta_{(j)}-\tdbeta_{(j)})\right\| 
> \epsilon/c \left|\, \bfX,\tdbeta\right.\right\}=o_p(1).
\end{equation}
Putting $\tilde{T}_j=L_j(\hbeta^*_{(j)}-\beta_{0(j)})$ and $\Delta_j^*=a_n(\tilde{T}_j-T^*_j)$, we arrive at
\begin{equation*}
\Prob\left\{\sup\|\Delta_j^*\| > \epsilon \mid \bfX,\tdbeta\right\}=o_p(1).
\end{equation*}
The desired conclusion \eqref{eq:Mnormweakconv} follows immediately as
$\nu[a_n \tilde{T}_j\mid \bfX,\tdbeta]=\nu[a_n T_j\mid \bfX]$.
\end{proof}

\begin{proof}[Proof of Proposition~\ref{prop:bth}]
By Assumption~\ref{as:design}, $\Prob(\bfX\in\scrX)\to 1$ with $\scrX$ defined in \eqref{eq:REX},
which then implies that RE$(q_0)$ holds with high probability.
Under Assumption~\ref{as:betasignal}, 
\begin{align*}
\inf_{j\in S_1}\|\beta_{0(j)}\| \gg \lmd s_0 \gg  \tau \quad\text{and}\quad
\sup_{j\in S_2}\|\beta_{0(j)}\|  \ll \frac{p_{\max}}{p_{\max}\vee\log J} \cdot \frac{\lmd}{s_0} \ll \tau,
\end{align*}
where $\tau$ is as in \eqref{eq:boundL2} with $q=q_0$.
Since $\lmd s_0  \gg b_{\supth} \gg  \lambda\surd{q_0}\asymp \tau$, it follows from \eqref{eq:boundL2} that 
\begin{equation*}
\Prob\left\{G(\tdbeta)=S_1\right\} \geq \Prob\left(\calE\cap \{\bfX\in\scrX\}\right) \to 1,
\end{equation*}
where $\calE$ is defined in Theorem~\ref{thm:grouplasso}.
Thus, only the strong coefficient groups in $S_1$ will be kept after thresholding.
Conditioning on the event $\{G(\tdbeta)=S_1\}$,
\begin{align*}
\sup_{j\in S_1}r_n\|\tdbeta_{(j)}-\beta_{0(j)}\|&= \sup_{j\in S_1}r_n\|\hbeta_{(j)}-\beta_{0(j)}\|=O_p (\sqpmax), \\
\sup_{j\in S_2}r_n\|\tdbeta_{(j)}-\beta_{0(j)}\|&= \sup_{j\in S_2}r_n\|\beta_{0(j)}\|=o_p(\sqpmax),
\end{align*}
where the first line comes from \eqref{eq:defrn} and the second from \eqref{eq:separablebeta} 
with $r_n=O(\sqn)$. 
This shows that $r_n(\tdbeta-\beta_0)\in\scrD(M_1)$ and completes the proof.
\end{proof}

\begin{proof}[Proof of Proposition~\ref{prop:GaussianC1}]
By Lemma~\ref{lm:NegahbanProp1}, with high probability, \eqref{eq:groupstrconvex} is satisfied with constants
$\kappa_1(\bmSigma)=\frac{1}{4}\Lambda_{\min}(\bmSigma^{1/2})$ and 
$\kappa_2(\bmSigma)\leq 9 \Lambda_{\max}(\bmSigma^{1/2})$; 
see the proof of Proposition 1 in \cite{Negahban12}. If $\Lambda_k(\bmSigma)\in[c_*,c^*]$ for all $k=1,\ldots,p$, 
we can find universal constants $(\kappa_1,\kappa_2)$ such that
\begin{equation*}
\liminf_{n\to\infty} \kappa_1(\bmSigma)>\kappa_1>0\quad\text{ and }\quad 
\limsup_{n\to\infty} \kappa_2(\bmSigma)<\kappa_2<\infty.
\end{equation*}
This shows that \eqref{eq:groupstrconvex} holds with 
universal constants $(\kappa_1,\kappa_2)$.
Put $\eps_1=p_{\max}/n$. Since $\Lambda_{\max}(\bmSigma)\leq c^*$ and $\eps_1\to 0$, 
\begin{equation*}
\Prob\left\{\sup_{j \in\N_J} \Lambda_{\max}(\bfC_{(jj)})\leq (1+\epsilon_2)^2 c^*\right\}
\geq 1- J e^{-n \epsilon_2^2/2} \to 1
\end{equation*}
for any given $\epsilon_2 \in(0,1)$, according to the proof of Proposition 2 in \cite{ZhangHuang08}. 
\end{proof}

\section{Estimation of error variance}\label{sec:unknownvar}

\subsection{Problem setup}

Let $\hsigma=\hsigma(y,\bfX)$ be an estimator of $\sigma$ when it is unknown. 
Recall that in our bootstrap method 
we draw $\varepsilon^* \mid \hsigma \sim \dnorm_n(0,\hsigma^2\bfI_n)$.
Put $U^*=\bfX^{\trans} \varepsilon^*/n$ and define $\sdelta$ as in \eqref{eq:defsdelta}.
We want to show that, with high probability, the conditional distribution $\nu[\sdelta\mid \tdbeta,\hsigma]$ 
is close to $\nu[\hdelta]$. In this case, the distribution of $U^*$ is different from that of $U$. 
To facilitate our analysis, let
\begin{equation}\label{eq:defsdelta0}
\sdelta_0\in \argmin_{\delta} V(\delta; \beta_0,(\sigma/\hsigma) U^*).
\end{equation}
Since $\nu[(\sigma/\hsigma) U^*]=\nu[U]$, we have $\nu[\sdelta_0 \mid \bfX]=\nu[\hdelta \mid \bfX]$. 
Our goal is then to bound $\|\sdelta-\sdelta_0\|$ for a large set of $(\tdbeta,\hsigma)$.
By triangle inequality,
\begin{equation*}
\|\sdelta-\sdelta_0\|\leq \|\sdelta-\shdelta \|+\|\sdelta_0-\shdelta\|,
\end{equation*}
where $\shdelta$ is as in \eqref{eq:defshdelta}.
The difference in the first term on the right is the result of using an estimated $\tdbeta$
instead of $\beta_0$ and can be bounded by Theorem~\ref{thm:finiteknownvar} 
conditioning on any $\hsigma>0$. 
In what follows, we will derive a bound for the second term, which
represents the effect of using a different noise level than the true $\sigma$ in the above simulation.
Proofs of all the results in this section can be found in Section~\ref{sec:apxproofs_sec3}.

\subsection{Main results}

For $\zeta>0$, define
\begin{align}
\scrS(\zeta) & = \{x >0: (x/\sigma) \vee (\sigma/x) \leq 1+ \zeta \}, \label{eq:sethsigma}\\
\calE^*(\zeta) &= \cap_{j=1}^J\{\|U^*_{(j)}\|\leq (1+\zeta)^{-1} w_j \lambda/2\}.
\end{align}

\begin{lemma}\label{lm:hsigma}
Assume that $|G(\beta_0)|\leq q$, RE$(q)$ is satisfied, and $\hsigma \in \scrS(\eta_3)$ for some $\eta_3>0$. 
Then on the event $\calE^*(\eta_3)$, we have
\begin{equation}\label{eq:XDeltaforsigma}
\frac{1}{n}\|\bfX(\sdelta_0-\shdelta)\|^2 \leq \frac{32w_{\max}r_n^2\lambda^2\eta_3}{w_{\min}\kappa^2(q)}\sum_{j\in A_0} w_j^2,
\end{equation}
where $\sdelta_0$ and $\shdelta$ are defined in \eqref{eq:defsdelta0} and \eqref{eq:defshdelta}, respectively.
\end{lemma}

For notational brevity, put 
\begin{equation}\label{eq:defh3}
h_3=\frac{8w_{\max}n\lambda^2\eta_3}{w_{\min}\kappa^2(q)}\sum_{j\in A_0} w_j^2
\end{equation}
so that the upper bound in \eqref{eq:XDeltaforsigma} is $4 r_n^2 h_3/n$. Assuming \eqref{eq:groupstrconvex}, we can bound $\|\sdelta_0-\shdelta\|$ conditional on $\hsigma$, in analogy to Theorem~\ref{thm:finiteknownvar}.

\begin{theorem}\label{thm:finiteunkvar}
Consider the model \eqref{eq:model} with $\varepsilon \sim \dnorm_n(0,\sigma^2\bfI_n)$ and let $J\geq 2$, $n \geq 1$. Assume that $|G(\beta_0)|\leq q$, RE$(q)$ holds, and \eqref{eq:groupstrconvex} is satisfied. 
Define $\sdelta_0$ and $\shdelta$ by \eqref{eq:defsdelta0} and \eqref{eq:defshdelta}, respectively, 
with $\varepsilon^*\mid \hsigma \sim \dnorm_n(0,\hsigma^2\bfI_n)$. 
Then on the event $\{\hsigma \in \scrS(\eta_3)\}$ for $\eta_3>0$, we have
\begin{equation}\label{eq:boundDeltaforhsigma}
\Prob\left[\|\sdelta_0-\shdelta\|^2\leq \frac{4r_n^2}{\kappa_1 n} \left.\left\{h_3+\kappa_2n\rho^2(2)h_1^2\right\} \,\right| \hsigma\right]\geq \Prob\left\{\calE^*(\eta_3) \mid \hsigma\right\},
\end{equation}
where $h_1$ and $h_3$ are defined in \eqref{eq:boundG2norm} and \eqref{eq:defh3}, respectively.
\end{theorem}


Under the asymptotic framework introduced in Section~\ref{sec:knownvar}, we obtain the main results
when an estimator $\hsigma^2$ of the unknown error variance is used in the simulation procedure.
Recall that $r_n$ is defined via \eqref{eq:defrn}.

\begin{theorem}\label{thm:asympestvar}
Consider the model \eqref{eq:model} with $\varepsilon \sim \dnorm_n(0,\sigma^2\bfI_n)$ 
and $p_{\max}/p_{\min}\asymp 1$. 
Suppose that Assumptions \ref{as:design} to \ref{as:pluginest} hold and \eqref{eq:asymcondhsigma} is satisfied.
Choose a suitable $\lambda \asymp \{(p_{\max} \vee \log J)/n\}^{1/2}$. 
Let $\sdelta$ and $\sdelta_0$ be any minimizers of \eqref{eq:defsdelta} and \eqref{eq:defsdelta0}, respectively, 
where $\varepsilon^* \mid \hsigma \sim \dnorm_n(0,\hsigma^2\bfI_n)$. 
Then for every $\epsilon>0$,
\begin{equation}\label{eq:DeltatoinP0_hsigma}
\Prob\left(\|\sdelta-\sdelta_{0}\| > \epsilon \mid 
\bfX,\tdbeta,\hsigma\right)=o_p(1).
\end{equation}
\end{theorem}

By the same argument, we can establish the conclusion in Corollary~\ref{thm:knownvarlaw} 
conditioning on $\hsigma$, which extends our inferential framework to the use
of an estimated noise level.

\subsection{Lasso as a special case}

Let $p_j=1$ for all $j$, $J=p$ and $q_0=s_0$. 
Consider a simpler situation with $S_1=A_0$ in Assumption~\ref{as:betasignal},
i.e., all active coefficients satisfy the beta-min condition. In this case, $\rho^2(2)=O(\log p/n)$ \eqref{eq:rho2}
and $h_2=0$ \eqref{eq:boundVforbeta}.
Under the assumptions of Theorems~\ref{thm:finiteknownvar} and \ref{thm:finiteunkvar}, 
we may combine \eqref{eq:finiteboundDelta} and \eqref{eq:boundDeltaforhsigma} to obtain
\begin{equation}\label{eq:boundlasso}
\Prob\left[\|\sdelta-\sdelta_0\|^2\leq \frac{8r_n^2}{\kappa_1 n}\left\{h_3+2\kappa_2n\rho^2(2)h_1^2\right\}
\defi B_1 \left|\, \bfX,\tdbeta,\hsigma\right.\right] \to 1
\end{equation}
in probability for the lasso.
Plugging in $h_1$, $h_3$ and that $\kappa^2(s_0)\geq \kappa_1/2$ due to Lemma~\ref{lm:GaussianRE}, 
the upper bound
\begin{equation*}
B_1=\left(c_1 \frac{s_0\log p}{n} + c_2 \eta_3\right) r_n^2 \lambda^2 s_0,
\end{equation*}
where $c_1$ and $c_2$ are positive constants. 
Under the scaling \eqref{eq:sparsescaling}, 
${s_0\log p}/{n}\ll {1}/{\sqn} =O(\eta_3)$,
since $\eta_3$ represents the convergence rate of $\hsigma$. Consequently,
$B_1=[c_2+o(1)] \eta_3 r_n^2 \lambda^2 s_0$.

Under comparable conditions, Theorem 4 in \cite{Zhou14} contains a similar bound for the lasso, 
\begin{equation*}
\Prob\left\{\|\sdelta-\sdelta_0\|^2\leq \frac{c_3 \eta_3 r_n^2 \lambda^2 s_0}{\phi_{\min}(c_4\phi_{\max}s_0)}
\defi B_2 \left|\, \bfX,\tdbeta,\hsigma\right.\right\} \to 1 \quad\text{in probability},
\end{equation*}
where $c_3, c_4>0$, $\phi_{\max}=\Lambda_{\max}(\bfC)$, and 
\begin{equation*}
\phi_{\min}(m)= \min_{1\leq |G(\Delta)|\leq m}\frac{\Delta^{\trans}\bfC\Delta}{\|\Delta\|^2}.
\end{equation*}
Note that here $G(\Delta)$ is the set of nonzero components of $\Delta\in\R^p$.
By definition $\phi_{\min}(m)=0$ for $m>n$. It is thus necessary to have $c_4\phi_{\max}s_0\leq n$
for $B_2$ to be of the same order as $B_1$. However,
by Lemma~\ref{lm:LmaxofC} below, 
$\phi_{\max}\asymp p/n$ with nearly full measure for column-normalized design matrices, 
and therefore, it is necessary that $s_0p=O(n^2)$. 
This severely limits the growth of the dimension $p$, as compared to the scaling of $s_0 \log p \ll \sqn$ 
for $B_1$. We see that, in addition to the substantial generalization to groups of unbounded size, 
this work also greatly improves the result in \cite{Zhou14} even for the special case of the lasso.

Let $\mbS^k$ be the $k$-dimensional unit sphere. The following lemma contains bounds for $\Lambda_{\max}(\bfC)$,
assuming that the columns of $\bfX$ are normalized so that $\|X_j\|=\sqn$ for all $j\in\N_p$.
 
\begin{lemma}\label{lm:LmaxofC}
Assume that $p>n$, $X_j/\sqn\in\R^n$ is drawn independently and uniformly over $\mbS^{n-1}$ for $j=1,\ldots,p$, and $\chi_n$ follows the $\chi$-distribution with $n$ degrees of freedom.
Then for every $\eps\in(0,{1}/{2})$,
\begin{align}\label{eq:Cmaxeigen}
&\Prob\left\{ 1\leq \frac{n}{p}\cdot\Lambda_{\max}(\bfC) \leq \left(\frac{1+(n/p)^{1/2}+\eps}{\theta_n-\eps}\right)^2\right\} \nonumber\\
&\quad\quad\quad\quad\quad\quad\quad\quad\quad
\geq 1-p\exp\left(-{n\eps^2}/{2}\right)-\exp\left(-{p\eps^2}/{2}\right), 
\end{align}
where $\theta_n=\E(\chi_n)/\sqn \in ({1}/{2},1]$ for all $n\geq 1$.
\end{lemma}

\begin{remark}
In this lemma, the distribution of $\bfX$ is assumed to be the uniform measure over
all $n\times p$ matrices with normalized columns. The result shows that 
the bounds for $\phi_{\max}=\Lambda_{\max}(\bfC)$ apply to most of such matrices, with nearly full measure
as $\exp({\epsilon^2}n/2)\gg p\to\infty$.
\end{remark}

A proof of Lemma~\ref{lm:LmaxofC} can be constructed 
using a similar argument as the one in the proof of Lemma 3.2 in \cite{Donoho06}
and thus is omitted here for brevity. 

\subsection{Proofs}\label{sec:apxproofs_sec3}

\begin{proof}[Proof of Lemma~\ref{lm:hsigma}]
This result can be obtained by a straightforward generalization of Lemma 8 in \cite{Zhou14}.
\end{proof}

\begin{proof}[Proof of Theorem~\ref{thm:finiteunkvar}]
For $\hsigma \in \scrS(\eta_3)$, we have
\begin{equation*}
\frac{1}{1+\eta_3} \leq \frac{\sigma}{\hsigma} \leq 1+ \eta_3,
\end{equation*}
and thus $\|(\sigma/\hsigma) U^*_{(j)} \| \leq (1+\eta_3) \|U^*_{(j)} \|$ for all $j$. Consequently, $\calE^*(\eta_3)$ implies that $\|(\sigma/\hsigma) U^*_{(j)} \| \leq w_j \lambda/2$ for $j=1,\ldots,J$ and therefore, by Theorem~\ref{thm:grouplasso}, $\|\sdelta_0\|_{\calG,2}\leq r_n h_1$. On the other hand, $\calE^*(\eta_3)$ obviously implies $\calE^*$ \eqref{eq:defcalE*} and thus $\|\shdelta\|_{\calG,2}\leq r_n h_1$. 
By \eqref{eq:groupstrconvex} with $\alpha=\alp^*=2$,
\begin{align*}
\kappa_1 \|\sdelta_0-\shdelta\|^2 & \leq \frac{1}{n}\|\bfX(\sdelta_0-\shdelta)\|^2 
+ \kappa_2 \rho^2(2) \|\sdelta_0-\shdelta\|^2_{\calG,2} \\
& \leq \frac{4 r_n^2 h_3}{n} + \kappa_2 \rho^2(2)(2 r_n h_1)^2,
\end{align*}
where \eqref{eq:XDeltaforsigma} and triangle inequality are used in the second line. Then \eqref{eq:boundDeltaforhsigma} follows since the above inequality holds on $\calE^*(\eta_3)$ for any $\hsigma \in \scrS(\eta_3)$.
\end{proof}

\begin{proof}[Proof of Theorem~\ref{thm:asympestvar}]
Following the same reasoning in the proof of Theorem~\ref{thm:asympknownvar}, all assumptions on $\bfX$ in Theorem~\ref{thm:finiteunkvar} are satisfied with high probability and $r_n^2\rho^2(2) h_1^2=o(1)$. A consequence of assumption \eqref{eq:asymcondhsigma} is that $\Prob(\{\hsigma \in \scrS(\eta_3)\})\to 1$ 
for some $\eta_3\ll\{q_0(p_{\max}\vee\log J)\}^{-1}$, which then implies
${r_n^2h_3}/{n}=O(r_n^2\lambda^2 q_0 \eta_3) =o(1)$.
By Theorem~\ref{thm:finiteunkvar}, with probability tending to one, 
\begin{equation*}
\Prob\left[\|\sdelta_0-\shdelta\|^2\leq \frac{4r_n^2}{\kappa_1 n}\left.\left\{h_3+\kappa_2n\rho^2(2)h_1^2\right\}=o(1) \,\right| \bfX,\hsigma\right]\geq \Prob\left\{\calE^*(\eta_3) \mid \bfX,\hsigma\right\}.
\end{equation*}
Since \eqref{eq:asymcondhsigma} implies that $\hsigma\to\sigma$ in probability, 
choosing a suitable $\lambda \asymp \{(p_{\max} \vee \log J)/n\}^{1/2}$ 
guarantees that $\Prob\left\{\calE^*(\eta_3) \mid \bfX,\hsigma\right\}\geq 1-2J^{1-a}\to 1$ with high probability. 
Thus, for every $\epsilon>0$,
\begin{equation}\label{eq:diff1}
\Prob(\|\sdelta_0-\shdelta\| > \epsilon/2 \mid \bfX,\hsigma)=o_p(1).
\end{equation}
Since Theorem~\ref{thm:asympknownvar} applies for every $\sigma>0$, 
it follows from \eqref{eq:DeltatoinP0} that
\begin{equation}\label{eq:diff2}
\Prob(\|\sdelta-\shdelta\| > {\epsilon}/{2} \mid \bfX,\tdbeta,\hsigma)=o_p(1).
\end{equation}
Now \eqref{eq:DeltatoinP0_hsigma} follows immediately from \eqref{eq:diff1} and \eqref{eq:diff2}. 
\end{proof}

\section{Proofs of results in Section~\ref{sec:generalizations}}
\label{sec:apxproofs}

\subsection{Error bounds for the block lasso}\label{sec:proofgeneralization}

Recall that $\alp^*$ is conjugate to $\alp\geq 2$. 
We first develop error bounds for the block lasso
$\hbeta$ \eqref{eq:blocklassodef} under a sub-Gaussian error.
That is, for any fixed $\|v\|_2=1$,
\begin{equation}\label{eq:subGaussian}
\Prob(|v^\trans \varepsilon| \geq x) \leq 2\exp\left(-\frac{x^2}{2\sigma^2}\right)\quad\text{for all } x>0,
\end{equation}
where $\sigma^2>0$ is a constant.
Define the operator norm for $\bfX_{(j)} \in \R^{n \times p_j}$, $j=1,\ldots,J$, by
\begin{equation*}
\|\bfX_{(j)}\|_{\alpha\to 2}= \max_{\|\theta\|_{\alpha}=1} \|\bfX_{(j)}\theta\|, \quad\quad \theta\in\R^{p_j}.
\end{equation*}

\begin{assumption}[Block normalization]\label{as:normalization}
We assume that 
\begin{equation}\label{eq:blocknormalization}
\frac{\|\bfX_{(j)}\|_{\alpha\to 2}}{\sqn} \leq (\bar{c})^{1/2} p_{\max}^{1/2-1/\alp} \defi \bar{c}(\alp) 
\quad \text{ for } j=1,\ldots,J,
\end{equation}
where $\bar{c}$ is the same constant as in Assumption~\ref{as:design}.
\end{assumption}

This assumption will be satisfied if $\Lambda_{\max}(\bfC_{(jj)})\leq \bar{c}$, since
\begin{equation*}
\frac{\|\bfX_{(j)}\|_{\alpha\to 2}}{\sqn}
=\max_{\theta\ne 0}\frac{\|\bfX_{(j)}\theta\|}{\sqn\|\theta\|} \cdot \frac{\|\theta\|}{\|\theta\|_{\alpha}} 
 \leq \left\{\Lambda_{\max}(\bfC_{(jj)})\right\}^{1/2}\cdot p_j^{1/2-1/\alp} \leq \bar{c}(\alp).
\end{equation*}
By Proposition~\ref{prop:GaussianC1}, if the rows of $\bfX$ are independent draws from $\dnorm_p(0,\bmSigma)$, 
then with high probability, $\Lambda_{\max}(\bfC_{(jj)})\leq \bar{c}$ and thus
Assumption~\ref{as:normalization} holds. This shows that it is indeed a reasonable and mild assumption.

The following theorem generalizes Theorem~\ref{thm:grouplasso} to the use of the $(1,\alp)$-group norm
for $\alp\in[2,\infty]$ based on the results in \cite{Negahban12}.

\begin{theorem}\label{thm:generalization}
Consider the model \eqref{eq:model} with sub-Gaussian error $\varepsilon$ \eqref{eq:subGaussian} and
$J\geq 2$. Let $\alpha \in [2,\infty]$ and $\frac{1}{\alpha}+\frac{1}{\alpha^*}=1$.
Suppose that $|G(\beta_0)|=q_0$ and that the design matrix $\bfX$ 
satisfies \eqref{eq:groupstrconvex} and Assumption~\ref{as:normalization}.
Then on the event 
$\calE_{\alpha}= \cap_{j=1}^J\{\|U_{(j)}\|_{\alpha^*} \leq \lambda/2\}$,
we have
\begin{align}
\|\hbeta-\beta_0\| & \leq \frac{6\lambda\surd{q_0}}{\kappa_1} \defi{\tau_{\alpha}}, \label{eq:L2normgroup_alpha}\\
\|\hbeta-\beta_0\|_{\calG,\alpha} & \leq \frac{24 \lambda q_0}{\kappa_1} \defi h_{\alp,1}, \label{eq:groupnorm_alpha}
\end{align}
for any $\hbeta$ defined by \eqref{eq:blocklassodef}, when
\begin{equation}\label{eq:minsize}
n> \frac{32\kappa_2}{\kappa_1} q_0 p_{\max}^{2/\alpha^*-1} \left\{(5p_{\max})^{1/2}+(4\log J)^{1/2} \right\}^2.
\end{equation}
Moreover, $\Prob(\calE_{\alpha})\geq 1-2J^{-1}$ if we choose
\begin{equation}\label{eq:lamdbaforalp}
\lambda \geq 4 \sigma p_{\max}^{1/\alpha^*-1/2}
\left\{\left(\frac{p_{\max}}{n}\right)^{1/2} + \left(\frac{\bar{c}\log J}{n}\right)^{1/2}\right\}.
\end{equation}
\end{theorem}

To prove this theorem, 
we establish in the following lemma an upper bound for $\rho(\alpha^*)$ \eqref{eq:rhostar}, which will be 
used in the proofs of a few other key results.

\begin{lemma}\label{lm:rhostar}
Let $J\geq 2$ and $p_{\max}\geq 1$. Then for $\alpha^*\in[1,2]$,
\begin{equation}\label{eq:boundrhostar}
\rho(\alpha^*)\leq p_{\max}^{1/\alpha^*-1/2} 
\left\{\left(\frac{5p_{\max}}{n}\right)^{1/2} + \left(\frac{4\log J}{n}\right)^{1/2}\right\}.
\end{equation}
\end{lemma}

\begin{proof}
We first bound $\rho(2)$. 
Let $V_1,\ldots,V_J$ be independent $\chi^2_d$ random variables with $d=p_{\max}\geq 1$, and put $V=\max_j V_j$.
Then by definition \eqref{eq:rhostar}, we have
\begin{equation*}
\sqn\rho(2) \leq \E(V^{1/2}) \leq (\E V)^{1/2}
\end{equation*}
due to the concavity of $(\cdot)^{1/2}$. By (4.3) in \cite{Laurent00}, for any $x>0$,
\begin{equation}\label{eq:chi2bound}
\Prob(V_1-d \geq 2\surd(d x) + 2x) \leq \exp(-x).
\end{equation}
For $x\geq d$, $2\surd(d x) + 2x \leq 4 x$ and thus
\begin{equation*}
\Prob(V_1-d \geq 4 x) \leq \Prob(V_1-d \geq 2\surd(d x) + 2x) \leq \exp(-x).
\end{equation*}
Equivalently, 
\begin{equation}\label{eq:chi2bound2}
\Prob(V_1-d \geq z) \leq  \exp(-z/4)
\end{equation}
for $z=4x \geq 4d$. Now for every $\theta\geq 0$ we have
\begin{align*}
\E(V)&  \leq 5d + \theta + J \int_{5d + \theta}^{\infty} \Prob(V_1\geq t) dt \\
& \leq 5d + \theta + 4 e^{-d} J \exp(-\theta/4),
\end{align*}
where \eqref{eq:chi2bound2} is used in the second inequality.
Taking $\theta = 4 \log J$ gives
\begin{equation*}
(\E V)^{1/2} \leq (5d + 4\log J + 4 e^{-d})^{1/2} \leq (5d)^{1/2} + (4\log J)^{1/2}
\end{equation*}
when $J\geq 2$ and thus,
\begin{equation}\label{eq:boundrho2}
\rho(2)\leq \left\{\left(\frac{5p_{\max}}{n}\right)^{1/2} + \left(\frac{4\log J}{n}\right)^{1/2}\right\}.
\end{equation}

For $\alpha^*\in[1,2]$ and $v\in\R^m$, 
\begin{equation*}
\frac{\|v\|_{\alpha^*}}{m^{1/\alpha^*}} \leq \frac{\|v\|}{{m}^{1/2}}.
\end{equation*}
Thus, by definition \eqref{eq:rhostar}
\begin{align*}
\rho(\alpha^*) &\leq  \E\left\{\max_{j=1,\ldots,J} (p_j)^{1/\alpha^*-1/2} \|Z_{(j)}\|/\sqn\right\} \\
& \leq p_{\max}^{1/\alpha^*-1/2} \rho(2),
\end{align*}
which with the bound \eqref{eq:boundrho2} implies \eqref{eq:boundrhostar}.
\end{proof}

\begin{proof}[Proof of Theorem~\ref{thm:generalization}]
The desired results follow from Corollary 1 in \cite{Negahban12}. Define
\begin{align*}
\calM& =\left\{\beta\in\R^p: \beta_{(j)}=0,\text{ for all } j \notin A_0\right\}, \\
\scrC_{\alpha} & = \left\{\Delta \in \R^p: \sum_{j\notin A_0}\|\Delta_{(j)}\|_{\alpha} 
\leq 3 \sum_{j\in A_0}\|\Delta_{(j)}\|_{\alpha}\right\}.
\end{align*}
We verify that the restricted strong convexity condition \citep{Negahban12} holds
over $\scrC_{\alpha}$. This amounts to verifying that, for $\Delta\in \scrC_{\alpha}$,
\begin{equation}\label{eq:RSC}
\frac{1}{n}\|\bfX\Delta\|^2 \geq \kappa_{\calL} \|\Delta\|^2,
\end{equation}
where $\kappa_{\calL}$ is a positive constant. Since
\begin{equation*}
\|\Delta\|_{\calG,\alpha} \leq 4 \sum_{A_0} \|\Delta_{(j)}\|_{\alpha} \leq 4 q_0^{1/2} \|\Delta\|
\end{equation*}
for $\Delta \in \scrC_{\alpha}$ and $\alp\geq 2$, it follows from \eqref{eq:groupstrconvex} that
\begin{equation*}
\frac{1}{n}\|\bfX\Delta\|^2 \geq [\kappa_{1} -16\kappa_2 q_0 \rho^2(\alpha^*) ]\|\Delta\|^2.
\end{equation*}
Simple algebra with \eqref{eq:boundrhostar} shows that $16\kappa_2 q_0 \rho^2(\alpha^*)<\kappa_1/2$ 
when \eqref{eq:minsize} holds, which leads to $\kappa_{\calL}=\kappa_1/2$ in \eqref{eq:RSC}.
All the other assumptions of Corollary 1 in \cite{Negahban12} can be verified as in the proof of their
Corollary 4. Then \eqref{eq:L2normgroup_alpha} and \eqref{eq:groupnorm_alpha} follow immediately 
by plugging in $\kappa_{\calL}=\kappa_1/2$ and the compatibility constant $\Psi(\calM)\leq \surd q_0$.
Thus we have shown that \eqref{eq:L2normgroup_alpha} and \eqref{eq:groupnorm_alpha} hold on 
the event $\calE_{\alpha}$. It remains to show the lower bound for $\Prob(\calE_{\alpha})$ with 
the choice of $\lambda$ in \eqref{eq:lamdbaforalp}. With some modifications of the proof of Lemma 5
in \cite{Negahban12}, one can show that
\begin{equation}\label{eq:chi2boundUj}
\Prob\left\{\frac{\|U_{(j)}\|_{\alp^*}}{\sigma} \geq \frac{2p_{\max}^{1-1/\alpha}}{\sqn} + \theta \right\}
\leq 2 \exp\left\{-\frac{n\theta^2}{2 \bar{c}(\alp)^2}\right\}
\end{equation}
for all $j\in\N_J$ and $\theta>0$ under Assumption~\ref{as:normalization}. Setting 
\begin{equation*}
\theta=2\bar{c}(\alp){(\log J/n)^{1/2}}=2 p_{\max}^{1/\alpha^*-1/2} {(\bar{c}\log J/n)^{1/2}}
\end{equation*}
and applying the union bound over all $j \in \N_J$, we have
\begin{align*}
\Prob(\calE_{\alp}^c) &\leq 
\Prob\left[\max_{j=1,\ldots,J}\|U_{(j)}\|_{\alp^*} \geq 
2 \sigma p_{\max}^{1/\alpha^*-1/2}\left\{\left(\frac{p_{\max}}{n}\right)^{1/2} + \left(\frac{\bar{c}\log J}{n}\right)^{1/2}\right\}\right] \\
& \leq 2J \exp(-2 \log J)=2J^{-1},
\end{align*}
which completes the proof.
\end{proof}

\subsection{Proofs of Theorems~\ref{thm:consistency_alpha} and \ref{thm:consistency_inf}}
\label{sec:proofalphacon}

Generalize the definition of $V(\delta; \beta_0, U)$ in \eqref{eq:defVn} to 
\begin{align}\label{eq:defValpha}
V_{\alpha}(\delta; \beta_0, U) &= \frac{n}{2r_n^2} \delta^{\trans} \bfC \delta - \frac{n}{r_n}\delta^{\trans}U\nonumber\\
&\quad+ n\lambda \sum_{j \in A_0} w_j\left(\| \beta_{0(j)}+ r_n^{-1}\delta_{(j)}\|_{\alp}- \|\beta_{0(j)}\|_{\alp}\right) 
+ \frac{n\lambda}{r_n} \sum_{j \notin A_0} w_j\|\delta_{(j)} \|_{\alp}, 
\end{align}
which is identical to the penalized loss function in \eqref{eq:blocklassodef}, up to an additive constant,
when $w_j=1$ for all $j$.
Suppose $\nu[\varepsilon^*]=\nu[\varepsilon]$. Letting $U^*=\frac{1}{n}\bfX^{\trans}\varepsilon^*$, we have
\begin{align}
\shdelta &=r_n(\hbeta^*-\beta_0) \in\argmin_{\delta} V_{\alp}(\delta;\beta_0,U^*), \label{eq:defshdelta_alp}\\
\sdelta &=r_n(\sbeta-\tdbeta) \in \argmin_{\delta} V_{\alp}(\delta;\tdbeta,U^*). \label{eq:defsdelta_alp}
\end{align}

We first bound the eigenvalues of the Hessian of the $\ell_\alpha$ norm, and then generalize
a few lemmas in Section~\ref{sec:overview} before our proof of Theorems~\ref{thm:consistency_alpha} 
and \ref{thm:consistency_inf}.

\begin{lemma}\label{lm:Hessian}
Let $f(x)=\|x\|_{\alpha}$ and $\bfH_f(x)$ be the Hessian of $f$ evaluated at $x=(x_j)_{1:m} \in \R^m$ when exists.
If $\alpha\in[2,\infty)$ and $x \ne 0$, then
\begin{equation}\label{eq:boundHf}
0 \leq \Lambda_{\min}\{\bfH_f(x)\} \leq \Lambda_{\max}\{\bfH_f(x)\} \leq \frac{\alpha-1}{\|x\|_{\alpha}}\cdot I(m>1).
\end{equation}
\end{lemma}

\begin{proof}
Let $z=(z_j)_{1:m}=(x_j |x_j|^{\alpha-2})_{1:m}$. 
For $\alp\geq 2$, straightforward calculations lead to
\begin{equation}\label{eq:gradoff}
\nabla f (x)= \|x\|_{\alpha}^{1-\alpha} z
\end{equation}
and second-order partial derivatives
\begin{align*}
\frac{\partial}{\partial x_j}(\|x\|_{\alpha}^{1-\alpha} z_j) 
& = (\alpha-1) \|x\|_{\alpha}^{1-2\alpha} \left(|x_j|^{\alpha-2} \|x\|_{\alpha}^{\alpha} - z_j^2 \right), \\
\frac{\partial}{\partial x_i}(\|x\|_{\alpha}^{1-\alpha} z_j) 
& = -(\alpha-1) \|x\|_{\alpha}^{1-2\alpha} z_i z_j,\quad i\ne j.
\end{align*}
Let $\bfD=\|x\|_{\alpha}^{-1}\diag(|x_1|,\ldots,|x_m|)$. Then for $x\ne 0$, 
\begin{equation}\label{eq:Hf}
\bfH_f(x)=\frac{\alpha-1}{\|x\|_{\alpha}} \left( \bfD^{\alpha-2}- \frac{z z^{\trans}}{\|x\|_{\alpha}^{2\alpha-2}}\right).
\end{equation}
Since $x_i\ne 0$ for some $i$, $\bfD \geq 0$ is positive semi-definite and $|x_j|/\|x\|_{\alpha}\leq 1$ for all $j$. 
By definition $zz^{\trans}\geq 0$. Therefore,
\begin{equation*}
\Lambda_{\max}(\bfH_f(x)) \leq \frac{\alpha-1}{\|x\|_{\alpha}}(\Lambda_{\max}(\bfD))^{\alpha-2} 
\leq \frac{\alpha-1}{\|x\|_{\alpha}}.
\end{equation*}
On the other hand, $\Lambda_{\min}(\bfH_f)\geq 0$ is a direct consequence of the convexity of
$f$. When $m=1$, $f(x)=|x|$ and $\bfH_f(x)=0$ for $x\ne 0$.
This completes the proof. 
\end{proof}

\begin{remark}
If $x_i,x_j=\pm 1$ for some $i\ne j$ and $x_k=0$ for all $k\ne i,j$, then 
\begin{equation*}
\Lambda_{\max}(\bfH_f(x))= 2^{2/\alpha -1} \cdot \frac{\alpha-1}{\| x\|_{\alpha}} 
\geq \frac{1}{2}\cdot \frac{\alpha-1}{\|x\|_{\alpha}}
\end{equation*}
for $\alpha\geq 2$. Thus, the upper bound in \eqref{eq:boundHf} can be decreased by 
no more than a factor of $2$.
\end{remark}

Recall that $M_2>0$ is a constant. For $\eta\in(0,1)$, define 
\begin{align}\label{eq:defA01_alp}
A_{01}=A_{01}(\eta)= \left\{j\in A_0: \frac{\|\beta_{0(j)}\|_{\alpha}}{\sqpj}> \frac{M_2}{r_n(1-\eta)}\right\}
\end{align}
and $A_{02}=A_0\setminus A_{01}$.
For $\alpha\in[2,\infty]$ define
\begin{align}
\eta_{\alpha,1}  & = \sup_{j \in A_{01}} \frac{\|\tdbeta_{(j)}-\beta_{0(j)} \|}{\| \beta_{0(j)}\|_{\alpha}},  
\label{eq:difftdbeta_alpha} \\
\eta_{\alpha,2} & = \sup_{j\in A_{01}} \frac{M_2\sqpj}{r_n \|\beta_{0(j)}\|_{\alpha}}, \label{eq:defrho_alpha}
\end{align}
and
\begin{equation}\label{eq:defB0_alpha}
\scrB_{\alpha,0}=\{\tdbeta\in\R^p:G(\tdbeta)\subset A_{01} \text{ and } \eta_{\alpha,1} \leq \eta\},
\end{equation}
which generalize the definitions of $\eta_1$, $\eta_2$, and $\scrB_0$ in Section~\ref{sec:overview}.
For $\tdbeta\in\scrB_{\alpha,0}$, we have $\eta_{\alpha,1}+\eta_{\alpha,2}<1$.
The following two lemmas bound the difference in $V_{\alp}$
when $\beta_0$ is replaced by $\tdbeta$.

\begin{lemma}\label{lm:boundVforbeta_alpha}  
Assume that $\alpha\in [2,\infty)$.  
Then for any $\tdbeta\in\scrB_{\alpha,0}$, $u\in \R^p$ and $\delta \in \scrD(M_2)$ 
as defined in \eqref{eq:defdeltaregion},
\begin{align}\label{eq:boundVforbeta_alpha}
&|V_{\alpha}(\delta; \tdbeta, u) - V_{\alpha}(\delta;\beta_0,u)| \nonumber\\
&\quad\quad\leq \frac{(\alpha-1)M_2 n \lambda (\eta_{\alpha,1}+\eta_{\alpha,2})}{r_n(1-\eta_{\alpha,1}-\eta_{\alpha,2})}
\sum_{A_{01}} w_j \sqpj I(p_j>1) + 2n \lambda\sum_{A_{02}} w_j \|\beta_{0(j)}\|_{\alpha}
\defi h_{\alpha,2}.
\end{align}
\end{lemma}

\begin{proof}
Put $\Dtdbeta=\tdbeta-\beta_0$ and 
\begin{equation*}
\nabla_j =( \| \tdbeta_{(j)}+ r_n^{-1}\delta_{(j)}\|_{\alpha}- \| \tdbeta_{(j)}\|_{\alp} )
-(\| \beta_{0(j)}+ r_n^{-1}\delta_{(j)}\|_{\alp}- \| \beta_{0(j)}\|_{\alp}).
\end{equation*}
It follows from $G(\tdbeta)\subset A_{01}\subset A_0$ \eqref{eq:defB0_alpha} that 
\begin{equation*}
V_{\alpha}(\delta; \tdbeta, u) - V_{\alpha}(\delta;\beta_0,u)=n \lambda \sum_{A_0} w_j \nabla_j.
\end{equation*}
First consider $j\in A_{01}$. The assumption $\eta_{\alpha,1}\leq \eta$ in \eqref{eq:defB0_alpha} implies
\begin{equation*}
\eta_{\alpha,1} + \frac{M_2 \sqpj}{r_n \|\beta_{0(j)} \|_{\alpha}} <1.
\end{equation*}
With \eqref{eq:difftdbeta_alpha} and $\|\delta_{(j)}\|\leq M_2\sqpj$ by \eqref{eq:defdeltaregion},
this leads to
\begin{equation}\label{eq:postivedifference_alpha}
\|\Dtdbeta_{(j)}\|+r_n^{-1}\|\delta_{(j)}\|<\|\beta_{0(j)}\|_{\alpha}\leq \|\beta_{0(j)}\|,
\end{equation}
since $\alp\geq 2$. 
Two direct consequences of \eqref{eq:postivedifference_alpha} are 
\begin{equation*}
\|\tdbeta_{(j)}\|\geq \|\beta_{0(j)}\|-\|\Dtdbeta_{(j)}\|>\|\delta_{(j)}\|/r_n
\end{equation*} 
and $\|\beta_{0(j)}\|>\|\delta_{(j)}\|/r_n$. Therefore, the function $f(x)=\|x\|_{\alpha}$, $x\in\R^{p_j}$, is differentiable 
in the ball centered at $\tdbeta_{(j)}$ or $\beta_{0(j)}$ with radius $\|\delta_{(j)}\|/r_n$. 
By the mean value theorem,
\begin{align}
\| \tdbeta_{(j)}+ r_n^{-1}\delta_{(j)}\|_{\alpha}- \| \tdbeta_{(j)}\|_{\alp}
& =  [\nabla f(\xi_j)]^{\trans} \frac{\delta_{(j)}}{r_n} \label{eq:gradfj}\\
\| \beta_{0(j)}+ r_n^{-1}\delta_{(j)}\|_{\alp}- \| \beta_{0(j)}\|_{\alp}
& =   [\nabla f(\xi^*_j)]^{\trans} \frac{\delta_{(j)}}{r_n}, \label{eq:gradfjstar}
\end{align}
where $\xi_j=\tdbeta_{(j)}+c_1 \delta_{(j)}/r_n$ and $\xi^*_j=\beta_{0(j)}+c_2 \delta_{(j)}/r_n$ 
for some $(c_1, c_2) \in (0,1)^2$. Let $\Delta\xi_{j}=\xi_j-\xi^*_j$ and 
$v_j=\nabla f(\xi_j)-\nabla f(\xi^*_j)$.
For any $c_3\in(0,1)$,
\begin{align*}
\|\xi^*_{j} + c_3 \Delta \xi_j\|& = \|\beta_{0(j)}+c_3\Dtdbeta_{(j)}+c_4r_n^{-1}\delta_{(j)}\| \\
& \geq \|\beta_{0(j)}\|-\|\Dtdbeta_{(j)}\|-r_n^{-1}\|\delta_{(j)}\|>0,
\end{align*}
where $c_4=c_2(1-c_3)+c_1c_3\in(0,1)$ and the last inequality follows from \eqref{eq:postivedifference_alpha}. 
Thus, the mapping $\nabla f: \R^{p_j}\to \R^{p_j}$ is differentiable at $\xi^*_{j} + c \Delta \xi_j$ for every $c \in (0,1)$;
see \eqref{eq:gradoff} and \eqref{eq:Hf}. By Theorem 5.19 in \cite{Rudin76}, there exits $c_3\in(0,1)$ such that
\begin{equation*}
\|v_j\| \leq \|\bfH_f(\xi^*_{j} + c_3 \Delta \xi_j)\Delta\xi_{j}\|
\leq \Lambda_{\max}(\bfH_f(\xi^*_{j} + c_3 \Delta \xi_j))\cdot\| \Delta\xi_{j}\|.
\end{equation*}
It then follows from Lemma~\ref{lm:Hessian} that 
\begin{align}
\|v_j\|& \leq \frac{(\alp-1)I(p_j>1)\|\Delta\xi_{j}\|}{\|\xi^*_{j} + c_3 \Delta \xi_j\|_{\alp}} \nonumber \\
& =(\alp-1)I(p_j>1)
 \frac{\|\Dtdbeta_{(j)}+(c_1-c_2)r_n^{-1}\delta_{(j)}\|}{\|\beta_{0(j)}+c_3\Dtdbeta_{(j)}+c_4r_n^{-1}\delta_{(j)}\|_{\alp}}.\label{eq:upperforvj_alp}
\end{align}
Triangle inequality applied to \eqref{eq:upperforvj_alp} with $c_i\in(0,1)$ and $\alp\geq 2$ gives
\begin{align*}
\|v_j\|& \leq (\alp-1)I(p_j>1)\frac{\|\Dtdbeta_{(j)}\|+r_n^{-1}\|\delta_{(j)}\|}{\|\beta_{0(j)}\|_{\alp}-\|\Dtdbeta_{(j)}\|-r_n^{-1}\|\delta_{(j)}\|}.
\end{align*}
Together with $\|\delta_{(j)}\|\leq M_2\sqpj$ and the definitions of $\eta_{\alp,1}$ and $\eta_{\alp,2}$,
we obtain the following bound for all $j\in A_{01}$,
\begin{align}\label{eq:vjbound}
|\nabla_j|& \leq \|v_j\| \|\delta_{(j)}\|/r_n \nonumber\\
& \leq  \frac{(\alp-1) M_2(\eta_{\alpha,1} + \eta_{\alpha,2})}{r_n(1-\eta_{\alpha,1}-\eta_{\alpha,2})}
\cdot I(p_j>1)\sqpj.
\end{align}
For $j\in A_{02}$, by definition \eqref{eq:defB0_alpha} we have $\tdbeta_{(j)}=0$ and thus
\begin{equation}\label{eq:djbound2}
|\nabla_j|\leq \| r_n^{-1}\delta_{(j)}-(\beta_{0(j)}+ r_n^{-1}\delta_{(j)})\|_{\alp} + \| \beta_{0(j)}\|_{\alp} 
=2\| \beta_{0(j)}\|_{\alp}
\end{equation}
by triangle inequality. The desired upper bound $h_{\alp,2}$ follows immediately by
combining \eqref{eq:vjbound} and \eqref{eq:djbound2}. 
\end{proof}

For $\alpha=\infty$, an additional assumption on the margin of $\beta_{0(j)}$
is needed to bound the difference between
$V_{\alp}(\delta; \tdbeta, u)$ and $V_{\alp}(\delta;\beta_0,u)$.
Following the definition in \eqref{eq:margin}, let $k^*(\theta)=\pi(1)$ if $d(\theta)>0$.
It is easy to show that for any $\|\Delta\|< d(\theta)$, 
\begin{equation*}
k^*(\theta+\Delta) =k^*(\theta)=\pi(1),\quad\quad
\sgn(\theta_{\pi(1)}+\Delta_{\pi(1)})=\sgn(\theta_{\pi(1)}),
\end{equation*}
and the function $f(x)=\|x\|_{\infty}$ is differentiable at $(\theta+\Delta)$:
\begin{equation}\label{eq:gradLinf}
\nabla f(\theta+\Delta)=\sgn(\theta_{\pi(1)}) e_{\pi(1)},
\end{equation} 
where $e_j$ is the $j$th unit vector in $\R^m$.
If $d(\beta_{0(j)})>0$ for all $j\in A_{01}$, define
\begin{align}
\eta_{4}  & = \sup_{j \in A_{01}} \frac{\|\tdbeta_{(j)}-\beta_{0(j)} \|}{d(\beta_{0(j)})},  \label{eq:defeta4}\\
\scrB_{\infty,1} & = \scrB_{\infty,0} \cap \{\tdbeta\in\R^p: \eta_{4} \in [0,1)\}. \label{eq:defscrD}
\end{align}

\begin{lemma}\label{lm:boundVforbeta_inf}  
Let $\alp=\infty$,  $\tdbeta\in\scrB_{\infty,1}$, and assume that   
\begin{align}
\inf_{j\in A_{01}} \frac{d(\beta_{0(j)})}{\sqpj} & > \frac{M_2}{r_n(1-\eta_4)}. \label{eq:margincond} 
\end{align}
Then for any $u\in \R^p$ and $\delta \in \scrD(M_2)$, we have
\begin{align}\label{eq:boundVforbeta_inf}
\left|V_{\infty}(\delta; \tdbeta, u) - V_{\infty}(\delta;\beta_0,u)\right| \leq 2n \lambda\sum_{A_{02}} w_j \|\beta_{0(j)}\|_{\infty}\defi h_{\infty,2}.
\end{align}
\end{lemma}

\begin{proof}
Let $j\in A_{01}$. Define $\xi_j=\tdbeta_{(j)}+c_1 \delta_{(j)}/r_n$, 
$\xi^*_j=\beta_{0(j)}+c_2 \delta_{(j)}/r_n$, and $f(x)=\|x\|_{\infty}$ for $x\in\R^{p_j}$.
We will show that $\nabla f(\xi_j)=\nabla f(\xi_j^*)$ for every $(c_1,c_2)\in(0,1)^2$.
Assumption~\eqref{eq:margincond} implies that 
\begin{equation*}
\eta_{4}d(\beta_{0(j)}) + \frac{M_2 \sqpj}{r_n} <d(\beta_{0(j)})
\end{equation*}
and consequently, 
\begin{equation}\label{eq:margininequality}
\|\tdbeta_{(j)}-\beta_{0(j)} \| + \|r_n^{-1} \delta_{(j)}\| <d(\beta_{0(j)}).
\end{equation}
It then follows that
\begin{equation*}
\|\xi_j^*-\beta_{0(j)}\| <\|r_n^{-1} \delta_{(j)}\| <d(\beta_{0(j)}).
\end{equation*}
Again by \eqref{eq:margininequality} we have $\|\tdbeta_{(j)}-\beta_{0(j)} \|<d(\beta_{0(j)})$ and
\begin{equation*}
\|\tdbeta_{(j)}+ r_n^{-1} \delta_{(j)}-\beta_{0(j)} \| <d(\beta_{0(j)}).
\end{equation*}
Then $\|\xi_j - \beta_{0(j)}\|< d(\beta_{0(j)})$
since $\xi_j$ lies between $\tdbeta_{(j)}$ and $\tdbeta_{(j)}+ r_n^{-1} \delta_{(j)}$.
By \eqref{eq:gradLinf} with $\theta=\beta_{0(j)}$, 
\begin{equation*}
\nabla f(\xi_j^*)=\nabla f(\xi_j)=\sgn(\theta_{\pi(1)}) e_{\pi(1)}.
\end{equation*}
Now we arrive at \eqref{eq:gradfj}, \eqref{eq:gradfjstar}, and $v_j=0$,
which implies $\nabla_j=0$ for all $j\in A_{01}$. Then \eqref{eq:boundVforbeta_inf} follows immediately
from \eqref{eq:djbound2} with $\alp=\infty$.
\end{proof}

Our next result is the counterpart of Lemma~\ref{lm:eventsubset}.
Define three events, $E^*_t$ and $E_2$ as in \eqref{eq:defE*t} and \eqref{eq:defE2}, and 
\begin{equation}\label{eq:defcalE*_alp}
\calE^*_{\alp} =\cap_{j=1}^J\{\|U^*_{(j)}\|_{\alp^*}\leq \lambda/2\}.
\end{equation}

\begin{lemma}\label{lm:eventsubset_alp}
Let $\alp\in[2,\infty]$ and $\alp^*$ be conjugate to $\alp$. Choose $\lambda$ as in \eqref{eq:lambdaasy_alp}.
Assume that
\begin{equation}\label{eq:aspforsubset}
q_0 p_{\max}^{2/\alp^*-1}(p_{\max} \vee \log J) \ll \sqn,
\end{equation} 
$r_n=O(\sqn)$, and $r_n^2h_{\alp,2}/n=o(p_{\min})$. Suppose that \eqref{eq:groupstrconvex} 
holds with universal constants $(\kappa_1,\kappa_2)$ when $n$ is large. 
Then for every $t\in(0,1)$, there is $N_1$ such that 
\begin{equation*}
\{\tdbeta \in \scrB_{\alp,0}\} \cap \calE^*_{\alp} \cap E^*_t \subset E_2
\end{equation*} 
when $n>N_1$.
\end{lemma}

\begin{proof}
With simple modifications, the proof follows closely to that of Lemma~\ref{lm:eventsubset}. 
The only difference is to show
\begin{equation}\label{eq:liminfsuff_alp}
\liminf_{n\to\infty} \left\{\kappa_1 \|\Delta\|^2 - \kappa_2 \rho^2(\alp^*) \|\Delta\|^2_{\calG,\alp} -4r_n^2h_{\alp,2}/n\right\} >0
\end{equation}
instead of \eqref{eq:liminfsuff}, where $\Delta=(1-c)(\sdelta-\shdelta)$ for some $c\in(0,1)$. 
It then suffices to show that the second term is $o(1)$.
Since $|G(\tdbeta)|\leq |A_0|= q_0$ for $\tdbeta\in\scrB_{\alp,0}$, it follows from Theorem~\ref{thm:generalization} that 
on $\calE^*_{\alp}$ we have $\|\sdelta\|_{\calG,\alp}\leq r_n h_{\alp,1}$ and 
$\|\shdelta\|_{\calG,\alp}\leq r_n h_{\alp,1}$, which implies
\begin{equation*}
\|\Delta\|_{\calG,\alp}\leq \|\sdelta\|_{\calG,\alp}+\|\shdelta\|_{\calG,\alp} \leq 2r_n h_{\alp,1}.
\end{equation*}
As $r_n=O(\sqn)$, 
\begin{align}\label{eq:rhorh_alp}
\rho(\alp^*) r_n h_{\alp,1}=O\left(q_0 p_{\max}^{2/\alp^*-1}(p_{\max} \vee \log J)/\sqn\right)=o(1)
\end{align}
by \eqref{eq:lambdaasy_alp}, the upper bound for $\rho(\alp^*)$ \eqref{eq:boundrhostar}, 
and the assumption \eqref{eq:aspforsubset}.
\end{proof}

We are now in a position to prove the main results.

\begin{proof}[Proof of Theorem~\ref{thm:consistency_alpha}]
First, we note that Lemma~\ref{lm:boundDeltaforbeta} holds with $h_{\alp,2}$ in place of $h_2$. 
That is, $\{\tdbeta\in\scrB_{\alp,0}\}\cap E_1^* \cap E_2$ implies
\begin{equation*}
\frac{1}{n}\|\bfX(\sdelta-\shdelta)\|^2 \leq 4 r_n^2 h_{\alp,2}/n.
\end{equation*}
Following a similar proof as that of Theorem~\ref{thm:finiteknownvar},
together with Lemma~\ref{lm:eventsubset_alp},
we can show that on the event $\{\tdbeta\in\scrB_{\alp,0}\}$,
\begin{equation}\label{eq:finiteboundDelta_alp}
\Prob\left[\|\sdelta-\shdelta\|^2
\leq \frac{4r_n^2}{\kappa_1 n}\left\{h_{\alp,2}+\kappa_2n\rho^2(\alp^*)h_{\alp,1}^2\right\} \left|\, \tdbeta\right.\right]\geq \Prob(\calE_{\alp}^* \cap E^*_t)
\end{equation}
when $n>N_1$.
Next, we show that $\eta_{\alp,2}=o(1)$ and $r_n^2h_{\alp,2}/n=o(1)$,
in place of \eqref{eq:upperrho} and \eqref{eq:boundh2} in the proof of  Theorem~\ref{thm:asympknownvar}.
By \eqref{eq:defrho_alpha} and \eqref{eq:condbeta0_alp}
\begin{align}\label{eq:etaalp2bound}
\eta_{\alp,2} 
\ll \frac{n^{1/2}}{r_n p_{\max}^{1/\alp^*} q_0 (p_{\max} \vee \log J)^{1/2}} \asymp \frac{1}{r_n q_0\lmd(p_{\max})^{1/2}},
\end{align}
where the second step is due to the choice of $\lmd$ \eqref{eq:lambdaasy_alp}.
It then follows from 
\eqref{eq:orderofr_n} that $\eta_{\alp,2} =o(q_0^{-1/2} p_{\max}^{-1})=o(1)$.
Similar to \eqref{eq:eta1and2}, $\eta_{\alp,1}=O(\eta_{\alp,2})=o(1)$ for $\tdbeta\in\scrB(M_1)$. 
Thus, from \eqref{eq:boundVforbeta_alpha} we have
\begin{align*}
\frac{r_n^2h_{\alp,2}}{n}  =O(r_n \eta_{\alp,2}q_0 \lambda  \sqpmax) 
+O\left(r_n^2 \lambda q_0 \sup_{S_2} \|\beta_{0(j)}\|_{\alp} \right) 
=o(1)
\end{align*}
by \eqref{eq:etaalp2bound} and \eqref{eq:condbeta0_alp}.
Furthermore, $r_n^2\rho^2(\alp^*) h_{\alp,1}^2=o(1)$ by \eqref{eq:rhorh_alp}.
Therefore, the upper bound for $\|\sdelta-\shdelta\|^2$ in \eqref{eq:finiteboundDelta_alp} is $o(1)$.
Then, one may show all the 
desired results by an essentially identical proof to 
that of Theorem~\ref{thm:asympknownvar}.
\end{proof}

\begin{proof}[Proof of Theorem~\ref{thm:consistency_inf}]
It suffices to verify \eqref{eq:margincond} when $n$ is large 
so that Lemma~\ref{lm:boundVforbeta_inf} may be applied.
By \eqref{eq:marginasy} and \eqref{eq:orderofr_n} with $p_{\max}\asymp p_{\min}$, 
\begin{equation*}
\inf_{j\in S_1} \frac{d(\beta_{0(j)})}{\sqpj}\gg \lambda\left(q_0/p_{\max}\right)^{1/2}=\Omega(1/r_n).
\end{equation*}
Then, it remains to show that $\eta_4=o(1)$. 
For $\tdbeta\in\scrB(M_1)$, 
\begin{equation*}
\sup_{j\in S_1}\|\tdbeta_{(j)}-\beta_{0(j)} \|=O\left(\sqpmax/r_n\right)=O(\lambda \surd{q_0}).
\end{equation*}
By definition \eqref{eq:defeta4},
\begin{align*}
\eta_4  =\sup_{j\in S_1}\frac{\|\tdbeta_{(j)}-\beta_{0(j)} \|}{d(\beta_{0(j)})} 
&\leq \sup_{S_1}\|\tdbeta_{(j)}-\beta_{0(j)} \| \cdot\sup_{j\in S_1}\frac{1}{d(\beta_{0(j)})} \\
& = O(\lambda \surd{q_0}) \cdot o((\lambda \surd{q_0})^{-1}) =o(1),
\end{align*}
where we have used assumption \eqref{eq:marginasy}.
\end{proof}

\section{Supplemental numerical results}\label{sec:numsupp}

\subsection{Results for individual inference}\label{sec:individualresults}

Let $r_A$, $r_I$ and $r$ denote the coverage rates
for active, zero, and all coefficients, respectively, and let $L_A$, $L_I$
and $L$ be the corresponding average interval lengths. 
Note that $1-r_I$ reports the type-I error rate of the test, $H_k: \beta_{0k}=0$, for $k\in \N_p$, 
and the power can be calculated 
by checking whether or not an estimated interval for an active coefficient covers zero.
Reported in Table~\ref{tab:singlesim} are the average results for each of the eight data generation settings
in the first simulation study.

\begin{table}[ht]
\caption{Power and confidence intervals for inference on individual coefficients\label{tab:singlesim}}
\begin{center}
\begin{tabular}{cclrrrrrrr}
  \hline  \hline
  \multicolumn{2}{c}{Data Setting} & \multirow{2}{*}{Method} & \multirow{2}{*}{PWR} & \multirow{2}{*}{$r_A$} & \multirow{2}{*}{$L_A$} & \multirow{2}{*}{$r_I$} & \multirow{2}{*}{$L_I$} & \multirow{2}{*}{$r$} & \multirow{2}{*}{$L$} \\
  \cline{1-2}
$(n,p)$ & $(a,d)$ &  & & & & & && \\
\hline
\multirow{12}{*}{$(100,200)$}
		&\multirow{3}{*}{(1, i)}  		&	bootstrap		& 54.0 & 43.5 & 0.456 & 96.7 & 0.136 & 94.0 & 0.152 \\    
		&						&	de-sparsified	& 67.5 & 76.5 & 0.540 & 89.3 & 0.542 & 88.6 & 0.542 \\   
		&						&	de-biased		& 60.0 & 81.0 & 0.566 & 98.8 & 0.568 & 97.9 & 0.568 \\  
		&\multirow{3}{*}{(1, ii)}		&	bootstrap		& 61.0 & 56.5 & 0.414 & 96.4 & 0.110 & 94.4 & 0.127 \\  
		&					 	&	de-sparsified	& 74.5 & 74.5 & 0.434 & 86.1 & 0.433 & 85.6 & 0.433 \\ 
		&						&	de-biased		& 68.0 & 83.0 & 0.443 & 98.8 & 0.441 & 98.1 & 0.441 \\   
		&\multirow{3}{*}{(2, i)}  		&	bootstrap		& 60.5 & 54.5 & 0.426 & 96.2 & 0.131 & 94.1 & 0.146 \\ 
		&						&	de-sparsified	& 70.0 & 79.0 & 0.498 & 85.1 & 0.503 & 84.8 & 0.503 \\  
		&						&	de-biased		& 60.0 & 95.0 & 0.519 & 98.9 & 0.525 & 98.7 & 0.525 \\  
		&\multirow{3}{*}{(2, ii)}		&	bootstrap		& 72.0 & 60.0 & 0.408 & 96.2 & 0.138 & 94.4 & 0.152 \\ 
		&					 	&	de-sparsified	& 79.5 & 85.0 & 0.422 & 82.8 & 0.420 & 82.9 & 0.420 \\ 
		&						&	de-biased		& 72.0 & 96.0 & 0.443 & 99.1 & 0.441 & 99.0 & 0.442 \\  
\hline
\multirow{12}{*}{$(100,400)$}
		&\multirow{3}{*}{(1, i)}  		&	bootstrap		& 47.0 & 35.5 & 0.335 & 96.2 & 0.065 & 94.7 & 0.071 \\    
		&						&	de-sparsified	& 65.0 & 74.0 & 0.537 & 85.0 & 0.539 & 84.8 & 0.539 \\   
		&						&	de-biased		& 53.0 & 73.5 & 0.523 & 99.1 & 0.526 & 98.5 & 0.526 \\  
		&\multirow{3}{*}{(1, ii)}		&	bootstrap		& 59.5 & 41.5 & 0.359 & 96.9 & 0.077 & 95.5 & 0.084 \\  
		&					 	&	de-sparsified	& 70.5 & 69.0 & 0.485 & 88.4 & 0.483 & 88.0 & 0.483 \\  
		&						&	de-biased		& 64.0 & 68.5 & 0.429 & 99.2 & 0.428 & 98.4 & 0.428 \\   
		&\multirow{3}{*}{(2, i)}  		&	bootstrap		& 62.0 & 46.5 & 0.416 & 96.5 & 0.095 & 95.3 & 0.103 \\ 
		&						&	de-sparsified	& 72.0 & 85.5 & 0.550 & 83.3 & 0.554 & 83.4 & 0.554 \\  
		&						&	de-biased		& 65.5 & 89.0 & 0.506 & 99.3 & 0.510 & 99.1 & 0.510 \\  
		&\multirow{3}{*}{(2, ii)}		&	bootstrap		& 71.0 & 40.5 & 0.374 & 96.8 & 0.099 & 95.4 & 0.106 \\ 
		&					 	&	de-sparsified	& 78.5 & 83.5 & 0.501 & 81.2 & 0.499 & 81.2 & 0.499 \\ 
		&						&	de-biased		& 73.0 & 85.0 & 0.445 & 99.2 & 0.444 & 98.8 & 0.444 \\  
\hline
\end{tabular}
\end{center}
PWR: power; $r$: coverage rate; $L$ interval length.
\end{table}

Largely consistent with the results in Table~\ref{tab:groupsim}, we see that our bootstrap method
shows a good control of the type-I error, implied by the observation that $r_I$, 
the coverage rate for zero coefficients, is slightly greater than but very close to $95\%$.
The $r_I$ for the de-sparsified lasso method is uniformly lower than $0.9$, dropping to $0.8$ in some cases,
implying that its type-I error rate is again substantially higher than the desired level of $5\%$.  
Although this might lead to some moderate degree of increase in power,
we argue that a strict control of false discoveries is critical for large-scale screening when $p$ is large
and $p\gg s_0$. For instance, the type-I error rate of the de-sparsified lasso is around $15\%$ for $n=100, p=400$.
This means that it brought about $0.1p=40$ more than expected false positives,
which was much larger than the number of true positives $s_0=10$. This would be a severe disadvantage
of the de-sparsified lasso method in the typical high-dimensional and sparse setting, $p\gg s_0$, 
under which the method was developed. The de-biased lasso, on the contrary, reached very high
coverage of zero coefficients, close to or above 99\% for all the cases. Its coverage of nonzero coefficients is
also seen to be higher than the other two methods.
However, the coverage rates for active coefficients of all three methods 
can be substantially lower than the desired level in many cases. 
For our method, this was caused by the inaccuracy in detecting active coefficients 
by thresholding a lasso estimate, which kept only 30\% to 40\% of them for the case $(n,p)=(100,400)$.
Even including the largest 20 components of the lasso would still identify  $<75\%$ of the true active variables.
Furthermore, without grouping the signal strength became small and
Assumption~\ref{as:betasignal} was not satisfied, 
as $\beta_{0k}$ was uniformly distributed over $(-1,1)$.

To give a concrete illustration of the effect of grouping variables,
for the setting $(n,p,a,d)=(100,400,1,\text{i})$ grouping improved the power of our method from $0.47$
to $1$ or $0.75$ and increased its $r_A$ from $35.5\%$ to $95\%$ or $85\%$, depending on the way of grouping. 
Grouping may also serve as a remedy for the observed low power
of the bootstrap method in individual inference. Given identified active groups, detection of individual nonzero coefficients
becomes an easier job due to the substantial dimension reduction.

\subsection{Running time}\label{sec:time}

We applied the de-sparsified lasso and our parametric bootstrap method with $N=300$ and $p_j=1$
on simulated data sets for a wide range of
combinations of $(n,p)$ in a high-dimensional setting with $p>n$. 
Reported in Table~\ref{tab:time} are the average running times per data set of the two methods 
for $10$ different combinations of $(n,p)$.
One sees that our method is uniformly faster across all scenarios and can be
significantly faster when the data size is large. 
For example, our bootstrap method is 12 times faster than the de-sparsified lasso when $n=400$ and $p=800$,
which corresponds to one order of magnitude improvement in terms of speed.

\begin{table}[h]
\centering
\caption{Average runtimes (in seconds) of de-sparsified lasso/bootstrap \label{tab:time}}
\vspace*{.1in}
\begin{tabular}{c|cccc}
  \hline  \hline
$n\setminus p$ & 100 & 200 & 400 & 800 \\ \hline
50	& 12.98/11.03	& 29.29/22.17	& 73.61/46.55	& 167.53/98.37  \\
100 	& 			& 41.19/20.72	& 108.62/42.54	& 233.26/83.80   \\
200 	&			&			& 223.30/43.86	& 496.69/91.54   \\
400	&			&			&			& 1242.78/102.33 \\
\hline
\end{tabular}
\end{table}

\bibliographystyle{asa}
\bibliography{sparsereferences,CBrefs}

\end{document}